\documentclass[11pt]{amsart}

\usepackage{amsmath,amssymb,latexsym,soul,cite,mathrsfs}

\usepackage{color,enumitem,graphicx}
\usepackage[colorlinks=true,urlcolor=blue,
citecolor=red,linkcolor=blue,linktocpage,pdfpagelabels,
bookmarksnumbered,bookmarksopen]{hyperref}

\usepackage[english]{babel}
\usepackage[left=2.9cm,right=2.9cm,top=2.8cm,bottom=2.8cm]{geometry}

\usepackage[colorinlistoftodos]{todonotes}
\makeatletter
\providecommand\@dotsep{5}
\def\listtodoname{List of Todos}
\def\listoftodos{\@starttoc{tdo}\listtodoname}
\makeatother

\numberwithin{equation}{section}

\numberwithin{equation}{section}
\newtheorem{thm}{Theorem}[section]
\newtheorem{cor}[thm]{Corollary}
\newtheorem{lem}[thm]{Lemma}

\newtheorem{defn}[thm]{Definition}

\newtheorem{exa}[thm]{Example}
\newtheorem{rmk}[thm]{Remark}

\newcommand{\N}{\mathbb{N}}

\newcommand{\Z}{\mathbb{Z}}

\newcommand{\X}{\mathbb{X}}
\newcommand{\Y}{\mathbb{Y}}
\newcommand{\W}{\mathbb{W}}
\def\R{\mathbb R}

\begin{document}

\title[\hfilneg \hfil On almost automorphic type solutions of abstract integral equations...]{On almost 
automorphic type solutions of abstract Integral equations, a Bohr-Neugebauer type property and some applications.}
\date{\today}

\author[A. Ch\'avez, M. Pinto \& U. Zavaleta \hfil \hfilneg]
{Alan Ch\'avez,  Manuel Pinto \& Ulices Zavaleta}  

\address{ Alan Ch\'avez \& Ulices Zavaleta\hfill\break
Departamento de Matem\'aticas \newline 
Facultad de Ciencias F\'isicas y Matem\'aticas\newline 
Universidad Nacional de Trujillo\newline 
Trujillo-Per\'u.}
\email{ \href{mailto:ajchavez@unitru.edu.pe, aulicesz@gmail.com}{ajchavez@unitru.edu.pe, aulicesz@gmail.com} }

\address{Manuel Pinto \hfill\break
Departamento de Matem\'aticas\newline 
Facultad de Ciencias\newline 
Universidad de Chile\newline 
Santiago-Chile.}
\email{\href{mailto:pintoj.uchile@gmail.com}{pintoj.uchile@gmail.com}}


%

\pretolerance10000


\begin{abstract}
\noindent In the present work we give some sufficient conditions to obtain a unique almost automorphic solution to abstract nonlinear integral equations which are simultaneously of advanced and delayed type and also a unique asymptotically almost automorphic mild solution to abstract integro-differential equations with nonlocal initial conditions, both situations are posed on Banach spaces. Also, we develop a Bohr-Neugebauer type result for the abstract integral equations. Before that, we introduce the notion of $\lambda$-bounded functions, develop the appropriate abstract theory and discuss the almost periodic situation. As applications, we study the existence of an asymptotically almost automorphic solution to integro-differential equations modeling heat conduction in materials with memory and also the existence of the almost automorphic solution to semilinear parabolic evolution equations with finite delay.
\end{abstract}


\maketitle

\section{Introduction}\label{intro}
It is well know that integral and integro-differential equations have taken a great interest due to their appearance  in several problems coming from pure mathematics as well as from the mathematical study of real life phenomena. In the last direction, we mention its applications in the following topics (among others): the heat conduction in materials with memory, transport phenomena in Biological Systems, optimal replacement problems in engineering and production economics, and so on, see for instance the references \cite{apMem,apBio,apEco}. Of course, this justifies the enormous progress in the qualitative and numerical study of them, for the numerical treatment, see the references \cite{SBAFPATFE,SWuSGa,VSAS,LWSWYY,LWYY2} (and references cited therein). In the context of diffentential, integral or integro-differential equations, it is common to start with the study of periodic problems (as a first instance), and then go accross its natural functional generalizations. Therefore, several works are devoted to the study of integral and integro-differential equations on the following complete function spaces: Periodic functions (with fixed period), almost periodic, pseudo-almost periodic, almost automorphic, pseudo almost automorphic and other related generalizations, see for instance \cite{SAbbas,JPCCC,TFuru,MNIslam,VKCLVP,Pint10} (and the references therein). In the mentioned works, the existence, uniqueness and stability of solution is treated. One observation is that in almost all the existing literature, only are studied integral equations which are of advanced or of delayed type, but there is no much work about equations that have a combination of  both situations. In the present work we study equations which are simultaneously of advanced and delayed type. 

Our work, has three main objectives. The first one is devoted to analyze and provide some sufficient conditions to ensure a unique almost automorphic solution of the following abstract integral equation of advanced and delayed type:
\begin{equation}\label{eq5}
y(t)=f(t,y(t),y(a_0(t)))+\int_{-\infty}^{t}C_1(t,s,y(s),y(a_1(s)))ds + \int_{t}^{+\infty}C_2(t,s,y(s),y(a_2(s)))ds\; ,
\end{equation}
and particularly setting $C_2\equiv 0$, to the equation
\begin{equation}\label{eq4}
y(t)=f(t,y(t),y(h_0(t)))+\int_{-\infty}^{t}C_1(t,s,y(s),y(h_1(s)))ds\; ;
\end{equation}
where for $i=1,2$ the functions $C_i$ are $\lambda_i$-bounded and Bi-almost automorphic kernels (see next section) and $f$ is an almost automorphic function in a concrete sense.

Our second objective, is provide the sufficient conditions for the existence of a unique asymptotically almost automorphic solution of the following integral equation:
\begin{equation}\label{eqNew}
y(t)= f(t,y(t),y(b_0(t)))+\int_0^t B_1(t,s,y(s),y(b_1(s)))ds + \int_t^{+\infty}B_2(t,s,y(s),y(b_2(s)))ds\; .
\end{equation}

Also, we study the existence of a unique asymptotically almost automorphic mild solution of the following nonautonomous and  integro-differential equation with nonlocal initial condition:
\begin{equation}\label{eqX}
u'(t)= A(t)u(t)+\int_0^t{B(t,s)u(s)ds + g(t,u(t))},\ t\geq0\; ,
\end{equation}
\begin{equation}\label{eqXX}
u(0)=u_0 + h(u),
\end{equation}
where $u_0 \in \X$, $A(t):D(A(t))\subset \X\rightarrow \X , \, \; \, t\in \R^+$ and $B(t,s):D(B(t,s))\subset \X\rightarrow \X ,\ t\geq s\geq  0$ are linear operators on the Banach space $\X$; $g(\cdot,\cdot)$  is an asymptotically almost automorphic function and $h$ satisfy some technical assumptions.

Finally, the third objective is to give some applications of our results: we provide applications to the heat conduction in materials with memory and also to semilinear parabolic evolution equations.



We also mention that, our work, is mainly motivated by the paper \cite{Pint10}, in which the author gives sufficient conditions to obtain a unique pseudo-almost periodic solution of integral equations of advanced and delayed type; also, our work has its roots in the paper \cite{16}.

On the other hand, recall that, under some conditions the Bohr-Neugebauer property for a differential equation, affirms that a bounded solution of an almost periodic differential equation is in fact almost periodic. This result extend the so called Massera's results in the periodic framework, which roughly asserts that, if a periodic linear systems of ordinary differential equations has a bounded solution, then it will have a periodic solution. In the present work we prove a slightly modified version in the direction of Bohr-Neugebauer for the integral equations (\ref{eq5}) and (\ref{eq4}). That is, if $y$ is a solution for (\ref{eq5}) or (\ref{eq4}) but with relatively compact range, then $y$ is almost automorphic. More precisely, under certain conditions we prove that (c.f Theorem \ref{Massteo}):

 {\sl  ``A solution of the integral equation (\ref{eq5}) (or (\ref{eq4})) is almost automorphic, if and only if, it has relatively compact range"}.

The Bohr-Neugebauer property, has been investigated for several kind of differential equations in the literature: ordinary and partial differential equations, differential equations with delay and also to functional differential equations, see for instance \cite{MAdiAKh,BrES} and references therein. To the best of our knowledge, there is no much results of this kind for integral equations.

The equations (\ref{eq5}) and (\ref{eq4}) appear naturally in concrete situations, for example (\ref{eq4}) can be obtained
 under suitable conditions as the mild solution of abstract Cauchy problems \cite{PazzyB}, while (\ref{eq5}) can be obtained as the solution of a neutral differential equation under conditions of exponential dichotomy \cite{Burt,BurtFuru,Pint100}, on the other hand equation (\ref{eqX}) concretely describe the dynamics of an important physical phenomena 
   \cite{apMem,SIvaPand} (see section \ref{appl}).
 
%

{ \bf Our paper is organized as follows:} In section \ref{prelim} we recover the notion of Bi-almost automorphic functions and give the definition of $\lambda$-bounded functions; also, we summarize and develop some results on almost automorphic functions and asymptotically almost automorphic functions that we need in the forthcoming sections; furthermore, we prove that those spaces are invariant under some integral operators. In section \ref{aasol} we study the existence and uniqueness of the almost automorphic solution to the integral equations (\ref{eq5}) and (\ref{eq4}). Section \ref{secB-N} is devoted to  our Bohr-Neugebauer's type result. In section \ref{aaasol} we analyze the existence and uniqueness of the asymptotically almost automorphic mild solution to equation (\ref{eqX})-(\ref{eqXX}). Finally, in section \ref{appl} we present applications to heat conduction in materials with memory and also to semilinear parabolic evolution equations with finite delay.

\section{Almost automorphic functions, asymptotically almost automorphic functions and some basic results.}\label{prelim}
\subsection{Definitions, notations and some known results.}

In the present paper, $\Z$ and $\R$ will denote the group of integer numbers and the field of real numbers respectively, $\R^+=[0,+\infty[$, $\R^-=]-\infty,0]$; while  $\X\, ,\Y$ and $\mathbb{W}$ are Banach spaces. The space of bounded and continuous functions from $\Y$ to $\X$ is denoted by $BC(\Y;\X)$, which is a Banach spaces under the norm of uniform convergence; and $C(\Y;\X)$ will denote the space of continuous functions from $\Y$ to $\X$.

We start with the definition of almost automorphic function given by Salomon Bochner \cite{5aaa,5aab,5aa,5bb}, who was the first mathematician that introduced and studied them. The definition is as follows
\begin{defn}\label{def1}
 A function $f\in BC(\R;\X)$ is said to be almost automorphic if given any
sequence $\{s_n'\}$ of real numbers, there exists a subsequence $\{s_n\}\subseteq\{s_n'\}$ and a function $\tilde f$,
such that the following pointwise limits holds: for each $t \in \mathbb{R}$,
$$\lim_{n\to\infty}f(t+s_n)=\tilde f(t)\; ,$$
and
$$\ \lim_{n\to\infty}\tilde f(t-s_n)=f(t)\; .$$
\end{defn}
Note that, in this definition, the limits ar taken only pointwise in $\mathbb{R}$; a strong version of this definition appear if we prefer {\sl uniform convergence} of the limits instead of pointwise convergence.  Of course, assuming uniform convergence  we carry out to the space of {\sl almost periodic} functions, which are also called {\sl Bochner almost periodic functions} or {\sl uniform almost periodic functions}, see \cite{39Cord,TDBook,29GM}. Let us denote by $AA(\R;\X)$ the space of almost automorphic functions from $\R$ to $\X$ and by $AP(\R;\X)$  the space of almost periodic functions from $\R$ to $\X$. 
%

The following theorem summarize some properties of almost automorphic functions, for a  proof and further properties  consult the references \cite{39Cord,TDBook,29GM}.

\begin{thm}
Let $f,g \in AA(\R;\X)$, then
\begin{enumerate}
\item For every $\alpha \in \R$, $f+\alpha g \in AA(\R;\X)$.
\item $AA(\R;\X)$ is a Banach space under the norm of uniform convergence in $\R$. That is, under the norm
$$||f||_{\infty}=\sup_{t\in \R}||f(t)||\, ,\, \; \, f \in AA(\R,\X)\, .$$
\item If $\tilde{f}$ is the function in definition $\ref{def1}$, then
$$||f||_{\infty}= ||\tilde{f}||_{\infty}.$$
\item $AP(\R;\X)$ is a closed subspace of $AA(\R;\X)$.
\item The range of $f$; i.e. $\mathcal{R}_f=\{f(t)\; :\; t\in \mathbb{R}\}$, is relatively compact in $\X$.
\end{enumerate}
\end{thm}

\begin{exa}
The classical example of an almost automorphic function which is not almost periodic is the following one: $\psi: \R \to \R$, defined by 
$$\psi(t)=\sin \left( \dfrac{1}{2+\cos(t)+\cos(\sqrt{2}t)} \right) \; .$$
\end{exa}
A natural generalization of the almost automorphic functions are given by the class of functions which are "almost automorphic at infinity", they are called "asymptotically almost automorphic" functions. In order to introduce this new class of functions, let us take account of the following definition.
\begin{defn}\label{def2}
 A function $f\in BC(\R\times\Y;\X)$ is said to be almost automorphic on bounded subsets of $\Y$, if given
any bounded subset $\mathcal{B}$ of $\Y$ and a sequence $\{s_n'\}$ of real numbers, there exists a subsequence
$\{s_n\}\subseteq\{s_n'\}$ and a function $\tilde f$, such that the following limits holds:
$$\lim_{n\to\infty}f(t+s_n,x)=\tilde f(t,x)\; ,$$
$$ \lim_{n\to\infty}\tilde f(t-s_n,x)=f(t,x)\; ,$$
where the limits are pointwise in $t\in\R\, $ and uniformly for $x$ in $\mathcal{B}$.
\end{defn}
We denote these class of functions by $AA(\R\times\Y;\X)$.
%
%
%
Now, Let us define the following function spaces 
$$C_0(\R^+;\X)=\Big{\{}\phi \in C(\R^+;\X): \lim_{t\to +\infty}||\phi(t)||=0 \Big{\}}\; ,$$
and
$$
C_0(\R^+\times \Y;\X)=\Big{\{} \phi \in C(\R^+\times\Y;\X): \lim_{t\to +\infty}||\phi(t,y)||=0,\  {\rm uniformly\ on\  bounded\ 
subsets\ of}\ \Y  \Big{\}}\; .$$
The following is the definition of asymptotically almost automorphic functions:

%

\begin{defn} A continuous function $g:\R^+\to\X$ (respectively $g:\R^+\times\Y\to\X$) is asymptotically almost
automorphic (respectively asymptotically almost automorphic in $t\in\R$, uniformly on bounded subsets of $\Y$) if
$g=f+\phi$, where $f \in AA(\R;\X)$ (respectively $f \in AA(\R\times\Y;\X)$) and $\phi \in C_0(\R^+;\X) $
(respectively $C_0(\R^+\times \Y;\X)$).
\end{defn}
For the asymptotically almost automorphic function $g=f+\phi$, the function $f$ is called the almost automorphic component, while the function $\phi$ is called its ergodic component. We denote by $AAA(\R^+;\X)$ the space of asymptotically almost automorphic functions and by $AAA(\R^+\times\Y;\X)$ the space of functions which are asymptotically almost automorphic in $t\in \R$ uniformly on bounded subsets of $\Y$.

In the space $AAA(\R^+;\X)$ we can define a norm: let $g\in AAA(\R^+;\X)$ with $g=f+\phi$, then  
\begin{equation}\label{NormAAA}
||g||:=\sup_{t\in\R}||f(t)|| +\sup_{t\in\R^+}||\phi(t)||\, .
\end{equation}

\noindent The following theorem summarize some properties of asymptotically almost automorphic functions (c.f. \cite{16})
\begin{thm}\label{TeoAAA}
We have 
\begin{enumerate}
\item The space $AAA(\R^+;\X)$ becomes a Banach space under the norm $(\ref{NormAAA})$.
\item $AAA(\R^+;\X)=AA(\R;\X)\oplus C_0(\R^+;\X)$. That is, the decomposition of an asymptotically almost automorphic function in its almost automorphic and ergodic components, is unique.
\item Let $g \in AAA(\R^+;\X)$, then its range is relatively compact in $\X$.
\end{enumerate}
\end{thm}


The following definition, which is the Bi-almost automorphicity, is a crucial ingredient in our approach. 

\begin{defn}\label{defBaa} A jointly continuous function $C:\R\times\R\times\X\times\Y\to \W$ is Bi-almost automorphic in $(t,s)\in \R\times \R$ uniformly for $(x,y)$ on bounded subsets of $\X\times\Y$ if 
given any sequence $\{s_n\}$ of real numbers and any bounded set $\mathcal{B}\subset \X\times\Y$, there exists a subsequence $\{s'_n\}\subseteq \{s_n\}$,
such that the function
$$\tilde C(t,s,x,y):=\lim_{n\to \infty}C(t+s'_n,s+s'_n,x,y)\, ,$$ 
is well defined for any $(x,y) \in \mathcal{B}$ and each $(t,s)\in \R\times\R$, and also we have the limit
$$\lim_{n\to \infty}\tilde C(t-s'_n,s-s'_n,x,y)=C(t,s,x,y)\, ,$$ 
for any $(x,y) \in \mathcal{B}$ and each $(t,s)\in \R\times\R$.
\end{defn}
This definition was early stated in \cite{TJXX} and further inspired the stochastic version \cite{ZCWLin}, see also \cite{TDBook}. We mention that, the discrete counterpart of the Bi-almost automorphicity was advanced in the following works of the first two authors \cite{ACh2,ACh3} see also \cite{MScTh}. In the cited works, the reader may found examples of Bi-almost automorphic functions in its continuous, stochastic and discrete versions (respectively).


%
%
Now we specify what we understand by a $\lambda$-bounded function. 
\begin{defn}
We say that a jointly continuous function $C:\R\times\R\times\X\times\Y\to \W$ is $\lambda$-bounded if there exist a positive function $\lambda:\R\times\R\to \R$ such that for every $\tau \in \R$ we have
$$||C(t+\tau,s+\tau,x,y)||_{\mathbb{W}}\leq \lambda(t,s),$$
where the inequality is uniform for $(x,y)$ on bounded subsets of $\X \times \Y$\,. 
%

\end{defn}

\begin{lem}\label{lemBou}
Let us suppose that the Bi-almost automorphic function $C:\R\times\R\times\X\times\Y\to \W$  is $\lambda$-bounded. 
Then, its limit function $\tilde{C}:\R\times\R\times\X\times\Y\to \W$ (see definition \ref{defBaa}) satisfies: 
$||\tilde{C}(t,s,x,y)||_{\W} \leq \lambda(t,s)$.
\end{lem}
\begin{proof}
Let $\mathcal{B}$ be a bounded subset of $\X \times \Y$. By the Bi-almost automorphicity of $C$, we have: given an arbitrary sequence of real numbers $\{s_n'\}$ there exist a subsequence $\{s_n\} \subset \{s_n'\}$ and a function $\tilde{C}$ such that the following pointwise limits in $(t,s)$ and uniform at $(x,y) \in  \mathcal{B}$, holds:
$$\tilde C(t,s,x,y):=\lim_{n\to \infty}C(t+s_n,s+s_n,x,y)\, ,$$ 
$$C(t,s,x,y)=\lim_{n\to \infty}\tilde C(t-s_n,s-s_n,x,y)\, .$$ 
On the other hand, we have
\begin{eqnarray}\label{Eq0.0}
||\tilde{C}(t,s,x,y)|| & \leq & ||\tilde{C}(t,s,x,y)-C(t+s_n,s+s_n,x,y)||+||C(t+s_n,s+s_n,x,y)|| \nonumber \\ 
&\leq & ||\tilde{C}(t,s,x,y)-C(t+s_n,s+s_n,x,y)||+\lambda(t,s).
\end{eqnarray}
Now, taking the limit as $n\to +\infty$ in the inequality (\ref{Eq0.0}), we obtain
$$||\tilde{C}(t,s,x,y)|| \leq \lambda(t,s)\, .$$
\end{proof}



\begin{defn}
We will say that  $F\in C_0^{\theta}(\R\times \R^+ \times \X \times \Y; \W)$, if  there exist a function \newline $\theta:\R \times \R^+ \to \R^+$ such that, 
$$||F(t,s,x,y)|| \leq \theta(t,s)\widehat{F}(s,x,y),\;  t\in\R,\; s \in \mathbb{R}^+,\;  x\in \X, \; y\in \Y ; $$
with  $\widehat{F} \in C_0(\R^+\times\X  \times\Y;\R^+)$.

\end{defn}

\subsection{Conditions}
The following are the basic abstract conditions that we impose in order to study equations (\ref{eq5}), (\ref{eq4}) and (\ref{eqNew}) :\\

\noindent {\bf (H1)} If $y$ is AA, then $y(a_i(\cdot))$ is AA, for $i=0,1,2$.

\noindent {\bf (H2)} If $y$ is AAA, then $y(b_i(\cdot))$ is AAA, for $i=0,1,2$.

\noindent 
{\bf (H3)} For $i=1,2$; the operators $C_i$ are Bi-almost automorphic in $(t,s)$ uniformly on bounded subsets of $\mathbb{X}\times \mathbb{Y}$ and are $\lambda_i$-bounded. Moreover, 
\begin{equation}\label{EqN01}
\sup_{t\in \mathbb{R}}\left( \int_{-\infty}^t\lambda_1(t,s)ds \right) = \alpha_1<+\infty \; ,\; \; \; \; \sup_{t\in \mathbb{R}} \left(\int_{t}^{+\infty}\lambda_2(t,s)ds\right)=\alpha_2<+\infty \; .
\end{equation}


\noindent {\bf (H4)} For $i=1,2$; let $\tilde{C}_i$ be the limit functions given in definition \ref{defBaa}. We assume that, $C_i$ are $(\mu_i, \tilde{\mu}_i)$-Lipschitz; that is, there exist functions $\mu_i, \tilde{\mu}_i:\mathbb{R}\times\mathbb{R}\rightarrow \mathbb{R}^+$ such that for $(u_1,u_2),(v_1,v_2)\in \mathcal{B}$, with $\mathcal{B}$ a bounded subset of $\mathbb{X}\times \mathbb{Y}$, we have
$$\left \|C_i(t,s,u_1,u_2)-C_i(t,s,v_1,v_2)\right \| \leq \mu_i(t,s)(\left \|u_1-v_1\right \|+\left \|u_2-v_2\right \|)\; ,$$
and 
$$\left \|\tilde{C}_i(t,s,u_1,u_2)-\tilde{C}_i(t,s,v_1,v_2)\right \| \leq \tilde{\mu}_i(t,s)(\left \|u_1-v_1\right \|+\left \|u_2-v_2\right \|)\; .$$
with 
$$\sup_{t\in \mathbb{R}}\left( \int_{-\infty}^t\mu_1(t,s)ds \right)=N_1<+\infty, \; \;\; \; \; \sup_{t\in \mathbb{R}}\left(\int_{t}^{+\infty}\mu_2(t,s)ds\right)=N_2<+\infty\; .$$
%

\noindent {\bf (H5)} For $i=1,2$, the functions $B_i$ have the decomposition $B_i=B_i^a+B_{i,0}^{\theta_i}$ in which $B_i^a$ are Bi-almost automorphic functions which satisfies condition {\bf (H3)}, 
and $B_{i,0}^{\theta_i}\in  C_0^{\theta_i}(\R\times \R^+ \times \X \times \Y; \W)$,
with 
$$\lim_{t\to \infty}\int_0^T \theta_1(t,s)ds=0,\; \; \; \forall \; T > 0\; ,$$
and 
$$\sup_{t\geq 0}\left( \int_{0}^{t}\theta_1(t,s)ds \right)=P_1<+\infty\; ,\; \; \; \; \;  \sup_{t\geq 0}\left( \int_{t}^{+\infty}\theta_2(t,s)ds \right)=P_2<+\infty\; .$$
Also, the Bi-almost automorphic functions $B_i^a$ are $(\nu_i, \tilde{\nu}_i)$-Lipschitz, 
 with 
\begin{equation}\label{Nweq2}
\sup_{t\geq 0}\left( \int_{-\infty}^t \nu_1(t,s)ds \right)=\beta_1<+\infty, \; \;\; \; \; \sup_{t\geq 0}\left(\int_t^{+\infty}\nu_2(t,s)ds \right)=\beta_2<+\infty \; ,
\end{equation}
and for every compact interval $[a,b] \subset \mathbb{R}$, the following limit holds
\begin{equation}\label{Nweq1}
\lim_{t\to \infty} \int_a^b\nu_1(t,s)ds=0\; .
\end{equation}

%
 
\noindent Finally, we also assume that there exists a function $\vartheta : \R \times \R \to \R^+$ such that  $|B_1^a(t,s,0,0)|\leq \vartheta(t,s)$, and 
$$\lim_{t\to \infty}\int_{-\infty}^0\vartheta(t,s)ds=0\; .$$

Note that the integral equations (\ref{eq5})-(\ref{eq4}), has the particular and especial case $C_i(t,s,u,v):= \Theta_i(t,s)f_i(s,u,v)$ (which in particular encodes the convolution situation $\Theta_i(t,s):=\Theta_i(t-s)$). In this case,  we impose the following condition

\noindent {\bf (E)} For $i=1,2$:

\begin{enumerate}
\item[(a)] The functions  $f_i: \R \times \X \times\Y \to \W$ are almost automorphic in $s$, uniformly on bounded subsets of $\mathbb{X}\times \mathbb{Y}$ and there exists constants $L_i=L(f_i)$, such that for all $s\in \mathbb{R}$ and for all  $(u_1,u_2),(v_1,v_2)\in \mathcal{B}$, with $\mathcal{B}$ a bounded subset of $\mathbb{X}\times \mathbb{Y}$ we have
$$\Big{|}\Big{|} f_i(s,u_1,u_2) - f_i(s,v_1,v_2)\Big{|}\Big{|}\leq L_i(||u_1-v_1||+||u_2-v_2||)\; .$$
\item[(b)] $\Theta_i:\R \times \R\to BC(\W,\W)$  are Bi-almost automophic kernels, uniformly on bounded subsets of $\W$; and  also they are $\lambda_i$-bounded with
\begin{equation}\label{EeQq01}
\sup_{t\in \mathbb{R}}\left( \int_{-\infty}^t\lambda_1(t,s)ds \right) <+\infty \; ,\; \; \; \; \sup_{t\in \mathbb{R}} \left(\int_{t}^{+\infty}\lambda_2(t,s)ds\right) <+\infty \; .
\end{equation}


\end{enumerate} 

Sometimes, the Lipschitz's constant $L_i$ given in condition {\bf (E)}-(a) is a function of the variable $s$ or a function of the radius $r$ of a closed ball in a Banach space (c.f. next section). 
%
%
%
 Obviously, it is possible to give conditions for the especial case $B_i(t,s,u,v)=\Upsilon_i(t,s)\tilde{B}_i(s,u,v)$, from where condition {\bf (H4)} can be deduced, we omit the details.

%
%
%

We mention that, the convolution situation of the equations treated here, i.e. the cases \newline
 $C_i(t,s,u,v) := \Theta_i(t-s)\widehat{C}_i(s,u,v)$ and $B_i(t,s,u,v)=\Upsilon_i(t-s)\tilde{B}_i(s,u,v) $ of the integral equations and the integro-differential equation (respectively), have been studied by the first author in  his Master thesis \cite{MScTh}. In this situation, the $\lambda_i$ functions are of the form $\lambda_i(t,s)=\lambda_i(t-s)$; and conditions in (\ref{EeQq01}) becomes: $\lambda_1 \in L^1(\mathbb{R}^+)$ and $\lambda_2 \in L^1(\mathbb{R}^-)$ (respectively).

%
%
%

%
%

\subsection{On the invariance of $AA(\mathbb{R};\mathbb{X})$ and $AAA(\mathbb{R^+};\mathbb{X})$ under some integral operators}
In order to study the equations of our interest, 
we need to develop some abstract lemmas in order to ensure that the spaces of almost automorphic and asymptotically almost automorphic functions are invariant under some integral operators, in this subsection we develop them.
\begin{defn}
Let $K$ be a compact subset of $\X$ and $T\subseteq \R$. A function $f$ is in the class $\mathcal{C}_K(T\times \X;\X)$ if  it 
satisfies: $f(t,\cdot)$ is uniformly continuous on $K$, uniformly for  $t\in T$.
\end{defn}
The proof of the following composition lemma is the same as \cite[Lemma 2.2]{16}.


\begin{lem}\label{l22} Let $x,y \in AAA(\R^+;\X), K=\overline{\{y(t):t\in\R\}}\times \overline{\{x(t):t\in\R\}}$ and $g\in AAA(\R^+\times \X\times \X;\X) \cap
\mathcal{C}_K(\R^+\times\X\times\X;\X)$, then $g(\cdot,x(\cdot),y(\cdot))\in AAA(\R^+;\X)$.
\end{lem}

Now, let us present our first result about integral operators that leaves invariant the almost automorphic function space.

\begin{lem}\label{lem12}
Suppose that condition {\bf (H1)} holds and let $C_1: \R\times \R\times \X\times \X \to \X$ be a Bi-almost automorphic function that satisfies conditions 
{\bf (H3)} and {\bf (H4)}. Then the integral operator $\Gamma $, such that:
$$\Gamma y(t)=\int_{-\infty}^{t}C_1(t,s,y(s),y(a_1(s)))ds\; ,$$
maps $AA(\R,\mathbb{X})$ into $AA(\R,\mathbb{X})$.
\end{lem}
\begin{proof}
Let $y\in AA(\mathbb{R};\mathbb{X})$ and define $x(t)=y(a_1(t))$; then, by hypothesis $x\in AA(\mathbb{R}; \mathbb{X})$. Let us take the bounded (and compact) set $\mathcal{B}=\overline{\{y(s):s\in\R\}}\times \overline{\{x(s):s\in\R\}}$. Since $C_1$ is jointly continuous and satisfies condition {\bf (H3)}, 
then the function
$$\Gamma y(t)=\int_{-\infty}^{t}C_1(t,s,y(s),y(a_1(s)))ds=\int_0^{+\infty}C_1(t,t-s,y(t-s),y(a_1(t-s)))ds\; ,$$
is a continuous function of $t$.  In fact, let $t_0$ be any real number;  because of the $\lambda_1$-boundedness of $C_1$, we have 
$$\int_0^{+\infty} \Big{|}\Big{|}C_1(t_0+h,t_0+h-s,y(t_0+h-s),x(t_0+h-s))-C_1(t_0,t_0-s,y(t_0-s),x(t_0-s))\Big{|}\Big{|}ds\leq $$
$$\leq 2 \int_0^{+\infty}\lambda_1(t_0,t_0-s)ds\; .$$
Therefore using (\ref{EqN01}) of the condition {\bf (H3)} and the Lebesgue's dominated convergence theorem we have
$$\lim_{h\to 0} ||\Gamma y (t_0+h)-\Gamma y (t_0)||=0\, ,$$
which implies that $\Gamma y$ is continuous at $t_0$.

Now let $\{s_n^{'}\}$ be an arbitrary sequence of real numbers, then there exist a subsequence $\{s_n\}\subset \{s_n^{'}\}$ and functions $\tilde{f},\tilde{y}$ and $ \tilde{x}$, such that the following pointwise limits holds
$$\lim_{n\to \infty} y(t+s_n)=\tilde{y}(t)\; ,\; \lim_{n\to \infty} \tilde{y}(t-s_n)=y(t)\; ,$$
$$\lim_{n\to \infty} x(t+s_n)=\tilde{x}(t)\; ,\; \lim_{n\to \infty }\tilde{x}(t-s_n)=x(t) \; .$$
Also, given 
any bounded subset $\mathcal{B}'\subset \mathbb{X}\times\mathbb{X}$ we have the following pointwise limits in $(t,s)$ and uniformly for $(x,y)\in \mathcal{B}'$ 
$$\lim_{n\to \infty}C_1(t+s_n,s+s_n,x,y)=\tilde C_1(t,s,x,y)\; ,$$
$$\lim_{n\to \infty}\tilde C_1(t-s_n,s-s_n,x,y)=C_1(t,s,x,y)\; .$$

\noindent  Let us take the bounded set $\mathcal{B}' =\mathcal{B}=\overline{\{y(s):s\in\R\}}\times \overline{\{x(s):s\in\R\}}$, and define 
$$\mho_n(t,s):=C_1(t+s_n,s+s_n,y(s+s_n),x(s+s_n))-\tilde C_1(t,s,\tilde y(s),\tilde x(s))\; .$$
Then, because of condition {\bf (H4)} we have

\begin{eqnarray*} 
||\mho_n(t,s)|| &\leq & ||C_1(t+s_n,s+s_n,y(s+s_n),x(s+s_n))- \tilde{C}_1(t,s,y(s+s_n), x(s+s_n))||\\
& +& ||\tilde{C}_1(t,s,y(s+s_n), x(s+s_n))-\tilde C_1(t,s,\tilde y(s),\tilde x(s))|| \\
&\leq &  ||C_1(t+s_n,s+s_n,y(s+s_n),x(s+s_n))- \tilde{C}_1(t,s,y(s+s_n), x(s+s_n))||\\
&+&\tilde{\mu}_1 (t,s)\Big{(}||y(s+s_n)-\tilde{y}(s)|| + ||x(s+s_n)-\tilde{x}(s)||\Big{)}\; .
\end{eqnarray*}
From this inequality and in the light of the previous limits, we are able to conclude the following pointwise limit:
$$\lim_{n\to +\infty}\mho_n(t,s)=0\; .$$
On the other hand, defining 
$$\mathcal{P}_n(t,s):=\tilde C_1(t-s_n,s-s_n,\tilde y(s-s_n),\tilde x(s-s_n))-C_1(t,s,y(s), x(s))\; ,$$
we obtain the following pointwise limit
%
$$\lim_{n\to +\infty}\mathcal{P}_n(t,s)=0\; .$$
%
Now, let us define the new function $\tilde{\Gamma} y$, by
$$\tilde \Gamma y(t):=\int_{-\infty}^{t} \tilde C_1(t,s,\tilde y(s),\tilde x(s))ds\; .$$
Note that by lemma \ref{lemBou} and {\bf (H3)} we conclude
\begin{eqnarray*}
||\tilde \Gamma y(t)||& \leq &\int_{-\infty}^{t} || \tilde C_1(t,s,\tilde y(s),\tilde x(s))||ds\\
&\leq &\int_{-\infty}^{t}\lambda_1(t,s)ds\\
 &\leq& \sup_{t \in \mathbb{R}}\int_{-\infty}^{t}\lambda_1(t,s)ds=\alpha_1<+\infty \; ,
\end{eqnarray*}
that is, $\tilde{\Gamma}y(t)$ is a well defined function.\\

From the $\lambda_1$-boundedness of $C_1$ and lemma \ref{lemBou}, we have that for every $n\in\N$ :
$$\left\| C_1(t+s_n,s+s_n,y(s+s_n),x(s+s_n))-\tilde C_1(t,s,\tilde y(s),\tilde x(s))\right\|\leq 2 \lambda(t,s)\; .$$

Therefore, using the Lebesgue's dominated convergence theorem we conclude the following pointwise limit in $t$
\begin{eqnarray*}
\lim_{n\to +\infty} \left \| \Gamma y(t+s_n)-\tilde \Gamma y(t)\right\| =0\; .
\end{eqnarray*}
That is, 
$$\lim_{n\to \infty}\Gamma y(t+s_n)=\tilde \Gamma y(t)\; .$$
With a similar process we obtain :
\begin{eqnarray*}
\lim_{n\to \infty}\tilde \Gamma y(t-s_n)=\Gamma y(t)\; .
\end{eqnarray*}
This completes the proof.
\end{proof}

Analogously, the following lemma holds:
\begin{lem}\label{lem13}
Suppose that condition {\bf (H1)} holds and let $C_2: \R\times \R\times \X\times \X \to \X$ be a Bi-almost automorphic function 
that satisfies condition {\bf (H3)} and {\bf (H4)}.
Then the integral operator $\Gamma $, defined by
$$\Gamma y (t)=\int_{t}^{+\infty}C_2(t,s,y(s),y(a_2(s)))ds\; ,$$
maps $ AA(\R;\X)$ into $ AA(\R;\X)$.
\end{lem}

For the particular case $C(t,s,y(s),x(b_1(s)))=\Theta(t,s)f(s,y(s),y(b_1(s)))$ which codify the convolution situation $C(t,s)=C(t-s)$, we have the following corollary of lemma \ref{lem12}
\begin{cor}
 Let $f \in AA(\R\times\X\times\X;\X)$ and  $\Theta : \mathbb{R}\times \mathbb{R}\to BC(\X;\X)$ be such that they satisfies conditions {\bf (E)}-$(a)$ and {\bf (E)}-$(b)$ respectively. 
 Then the integral operator $\Gamma $, such that :
$$\Gamma y (t)=\int_{-\infty}^{t}\Theta(t,s)f(s,y(s),y(a_1(s)))ds\; ,$$
maps $ AA(\R;\X)$ into $ AA(\R;\X)$.
\end{cor}

In the same way, the analogous corollary to the previous one, but for lemma \ref{lem13} can be deduced, we omit the details. Now, since the space of almost automorphic functions is a vector space, we conclude the following corollary:

\begin{cor}\label{CorAA}
Suppose that condition {\bf (H1)} holds and $C_1,C_2 :\R\times \R\times \X\times \X \to \X$ are Bi-almost automorphic  functions that satisfies conditions {\bf (H3)} and {\bf (H4)}. Then, the integral operator $\Gamma $, defined by
$$
\Gamma y (t)=\int_{-\infty}^{t}C_1(t,s,y(s),y(a_1(s)))ds+\int_{t}^{+\infty}C_2(t,s,y(s),y(a_2(s)))ds\; ,$$
maps $ AA(\R;\X)$ into $ AA(\R;\X)$.
\end{cor}
The following lemma is needed for the study of asymptotically almost automorphic solutions of the integral  equation (\ref{eqNew}), in particular it is of interest for the  integro-differential equation (\ref{eqX})-(\ref{eqXX}).

\begin{lem}\label{l23} Suppose condition {\bf (H2)} holds and, for $i=1,2$ the functions $B_i:\mathbb{R}\times \mathbb{R}\times \X \times \X \to \X$ 
satisfies condition {\bf (H5)}. 
 Then,  
the integral operators $F_1$ and $F_2$, such that
$$ F_1 y(t)=\int_0^{t} B_1(t,s,y(s),y(b_1(s)))ds\; ,$$
$$ F_2 y(t)=\int_t^{\infty} B_2(t,s,y(s),y(b_2(s)))ds\; ;$$
maps $AAA(\mathbb{R}^+; \X)$ into $AAA(\mathbb{R}^+; \X)$ .
\end{lem}
\begin{proof}
First let us prove that the operator $F_1$ leaves invariant the space $AAA(\mathbb{R}^+; \X)$. Let $x(s)=y(b_1(s))$; then by  hypothesis if $y=y^a+y^0$ is in $AAA(\mathbb{R}^+; \X)$, then  $x=x^a+x^0$  belongs to $AAA(\mathbb{R}^+; \X)$. Also, from hypothesis we have that $B_i=B_i^a+B_{i,0}^{\theta_i}$, where $i=1,2$. Therefore


\begin{eqnarray*}
F_1 y (t)&=& \int_0^t B_1(t,s,y(s),x(s))ds=\int_0^t \left( B_1^a(t,s,y(s),x(s))+B_{1,0}^{\theta_1}(t,s,y(s),x(s))\right) ds\\
&=& \int_0^t \Big{(} B_1^a(t,s,y(s),x(s))-B_1^a(t,s,y^a(s),x^a(s)) \Big{)}ds +\\
&+& \int_0^t  B_1^a(t,s,y^a(s),x^a(s)) ds + \int_0^t B_{1,0}^{\theta_1}(t,s,y(s),x(s))ds \\
		 &=& \int_{-\infty}^t  B_1^a(t,s,y^a(s),x^a(s))ds -\int_{-\infty}^0 B_1^a(t,s,y^a(s),x^a(s))ds +\\
		 &+&\int_0^t \Big{(} B_1^a(t,s,y(s),x(s))-B_1^a(t,s,y^a(s),x^a(s)) \Big{)} ds+\int_0^t B_{1,0}^{\theta_1}(t,s,y(s),x(s))ds\\
		 &=:& J_1(t)+J_2(t)\; ,
\end{eqnarray*}
where: 
$$J_1(t):=\int_{-\infty}^t  B_1^a(t,s,y^a(s),x^a(s))ds\; ,$$
and
\begin{eqnarray*}
J_2(t)&:=&\int_0^t \Big{(} B_1^a(t,s,y(s),x(s))-B_1^a(t,s,y^a(s),x^a(s)) \Big{)} ds-
 \int_{-\infty}^0 B_1^a(t,s,y^a(s),x^a(s))ds+\\
 &+& \int_0^t B_{1,0}^{\theta_1}(t,s,y(s),x(s))ds\; .
\end{eqnarray*}
We claim that $J_1 \in AA(\R;\X)$ and $J_2 \in C_0(\R^+;\X)$.\\
The assertion that $J_1 \in AA(\R;\X)$ is clear from  conditions in {\bf (H5)} and Lemma \ref{lem12}. We will have the second affirmation into three steps:\\

\noindent {\bf The first step:} for the first integral in $J_2$, we have: 

$$\int_0^t \Big{|}\Big{|}B_1^a(t,s,y(s),x(s))-B_1^a(t,s,y^a(s),x^a(s))\Big{|}\Big{|}ds\leq \int_0^t \nu_1(t,s)\Big{(}||y^0(s)|| +||x^0(s)||\Big{)}ds\; .$$
Now, since 
$y^0,x^0 \in C_0(\R^+; \X)$ we have: given $\epsilon >0$, there exist $T>0$ such that if $s>T$, then 
$||y^0(s)||<\dfrac{\epsilon}{2}$ and $ ||x^0(s)||<\dfrac{\epsilon}{2}$. Therefore,  for $t>0$ big enough  and using condition {\bf (H5)}, we have

\begin{eqnarray*}
\int_0^t ||B_1^a(t,s,y(s),x(s))-B_1^a(t,s,y^a(s),x^a(s))||ds &\leq& \int_0^T \nu_1(t,s)\Big{(} ||y^0(s)||+||x^0(s)||\Big{)} ds\\
&+&\int_T^t \nu_1(t,s)\Big{(} ||y^0(s)||+||x^0(s)||\Big{)}ds\\
&\leq& \Big{(} ||y^0||_{\infty}+||x^0||_{\infty} \Big{)} \int_0^T \nu_1(t,s)ds\\
&+& \epsilon \int_T^t \nu_1(t,s)ds\\
&< &(||y^0||_{\infty}+||x^0||_{\infty}+\beta_1)\epsilon\; .
\end{eqnarray*}

\noindent {\bf The second step:}  for the second integral, and using conditions {\bf (H5)} again, we have: given $\epsilon >0$, there exist $T_0>0$, such that if $t\geq T_0$, then by (\ref{Nweq1}) we have
$$ \int_{-(n+1)}^{-n}\nu_1(t,s)ds<\frac{\epsilon}{2^{n+1}}\; , \; \; \forall n \in \mathbb{N}\; ; $$
and also 
$$\int_{-\infty}^0\vartheta(t,s)ds <\epsilon\; .$$
Therefore, for $t\geq T_0$ we obtain :

\begin{eqnarray*}
||\int_{-\infty}^0 B_1^a(t,s,y(s),x(s))ds||&\leq&\int_{-\infty}^0 ||B_1^a(t,s,y^a(s),x^a(s))||ds \\
& \leq &\int_{-\infty}^0 ||B_1^a(t,s,y^a(s),x^a(s))-B_1^a(t,s,0,0)||ds\\
&+&\int_{-\infty}^0 ||B_1^a(t,s,0,0)||ds\\
&\leq & \int_{-\infty}^0\nu_1(t,s) \Big{(}||y^a(s)||+||x^a(s)||\Big{)}ds+ \int_{-\infty}^0 ||B_1^a(t,s,0,0)||ds\\
&\leq & \Big{(}||y^a||_{\infty}+||x^a||_{\infty}\Big{)}\int_{-\infty}^0\nu_1(t,s)ds+ \int_{-\infty}^0 ||B_1^a(t,s,0,0)||ds\\
&\leq & \Big{(}||y^a||_{\infty}+||x^a||_{\infty}\Big{)} \left( \sum_{n=0}^{\infty} \int_{-(n+1)}^{-n}\nu_1(t,s)ds \right)+ \int_{-\infty}^0 \vartheta(t,s)ds\\
& < &(||y^a||_{\infty}+||x^a||_{\infty}+1)\epsilon.
\end{eqnarray*}

\noindent {\bf The third step:} In this step we want to prove the following limit 
$$\lim_{t\to +\infty}\int_0^t B_{1,0}^{\theta_1}(t,s,y(s),x(s))ds=0\; .$$
In fact, from condition {\bf (H5)}, we have:  
$$||B_{1,0}^{\theta_1}(t,s,x,y)||\leq \theta_1(t,s)\widehat{B}_{1,0}(s,x,y)\, ,$$
where $\widehat{B}_{1,0} \in C_0(\mathbb{R}^+ \times \X \times \X; \mathbb{R}^+)$, 
and 
$$\lim_{t\to \infty}\int_0^T\theta_1(t,s)ds=0\, ,\; \; \forall T>0\, .$$
Therefore, given $\epsilon >0$, there exists $T_1>0$, such that if $t>T_1$, then
\begin{eqnarray*}
||\int_0^t B_{1,0}^{\theta_1}(t,s,y(s),x(s))ds\;|| &\leq&  \int_0^t|| B_{1,0}^{\theta_1}(t,s,y(s),x(s))||ds \\
 &\leq &  \int_0^t\theta_1(t,s)\widehat{B}_{1,0}(s,y(s),x(s))ds\\
 &\leq & \int_0^{T_1}\theta_1(t,s)\widehat{B}_{1,0}(s,y(s),x(s))ds+ \int_{T_1}^t\theta_1(t,s)\widehat{B}_{1,0}(s,y(s),x(s))ds\\
 &\leq & \sup_{s\in [0,+\infty[} \left( \widehat{B}_{1,0}(s,y(s),x(s)) \right) \int_0^{T_1}\theta_1(t,s)ds\\
 &+& \int_{T_1}^t\theta_1(t,s)\widehat{B}_{1,0}(s,y(s),x(s))ds\\ 
 & < & \left( \sup_{s\in [0,+\infty[} \left( \widehat{B}_{1,0}(s,y(s),x(s)) \right)+  \int_{T_1}^t\theta_1(t,s)ds \right)\epsilon\\
 & < & \left( \sup_{s\in [0,+\infty[} \left( \widehat{B}_{1,0}(s,y(s),x(s)) \right)+ P_1 \right)\epsilon\; .
\end{eqnarray*} 
This proves that $F_1$ is an operator that leaves invariant the space $AAA(\R^+; \X)$.\\

\noindent Now, let us prove that the operator $F_2$ leaves invariant the space $AAA(\R^+; \X)$. Let us define again $x(s)=y(b_2(s))$, and take $y=y^a+y^0$ in $AAA(\mathbb{R}^+; \X)$; then by hypothesis the new function $x=x^a+x^0$ is also in $AAA(\mathbb{R}^+; \X)$. Therefore, we have 

\begin{eqnarray*}
\int_t^{+\infty} B_2(t,s,y(s),x(s))ds &=&\int_t^{+\infty}\left( B_2^a(t,s,y(s),x(s))+B_{2,0}^{\theta_2}(t,s,y(s),x(s))\right) ds\\
&=&\int_t^{+\infty}\Big{(}  B_2^a(t,s,y(s),y(x(s)))-B_2^a(t,s,y^a(s),x^a(s)) \Big{)} ds +\\
&+& \int_t^{+\infty} B_2^a(t,s,y^a(s),x^a(s))ds + \int_t^{+\infty}B_{2,0}^{\theta_2}(t,s,y(s),x(s)) ds\\
&=& J_3(t)+J_4(t)\; ,
\end{eqnarray*}
where
$$J_3(t):= \int_t^{+\infty}B_2^a(t,s,y^a(s),x^a(s)))ds\; ,$$
and 
$$J_4(t):= \int_t^{+\infty} \Big{(} B_2^a(t,s,y(s),x(s))-B_2^a(t,s,y^a(s),x^a(s))\Big{)}ds+ \int_t^{+\infty}B_{2,0}^{\theta_2}(t,s,y(s),x(s)) ds\; .$$
The integral $J_3(t)$ is almost automorphic because of the conditions in {\bf (H5)} and lemma \ref{lem13}.  Now let us prove that  
$$\lim_{t \to +\infty}J_4(t)=0\; .$$ 
We proceed in two steps:\\

\noindent {\bf The first step:} For the first integral of $J_4$, we have
\begin{eqnarray*}
\Big{|} \Big{|} \int_t^{+\infty} \Big{(} B_2^a(t,s,y(s),x(s))-B_2^a(t,s,y^a(s),x^a(s))\Big{)}ds\Big{|} \Big{|} \leq \int_t^{+\infty}\nu_2(t,s)\Big{(} ||y^0(s)||+||x^0(s)|| \Big{)}ds\, .
\end{eqnarray*}
Since $y^0, x^0 \in C_0(\R^+;\X)$, then given $\epsilon >0$, there exists $T_2>0$ such that for every $t\geq  T_2$  we obtain $||y^0(t)||+||x^0(t)||<\epsilon$,  now using condition (\ref{Nweq2}) in {\bf (H5)} we conclude that $||J_4(t)||< \epsilon\, \beta_2$ for every $t\geq T_2$ .\\

\noindent {\bf The second step:} In order to estimate the second integral of $J_4$, we use condition {\bf (H5)} as follows:
$$||B_{2,0}^{\theta_2}(t,s,x,y)||\leq \theta_2(t,s)\widehat{B}_{2,0}(s,x,y)\, ,$$
where $\widehat{B}_{2,0} \in C_0(\mathbb{R}^+ \times \X \times \X; \mathbb{R}^+)$;  then, given $\epsilon >0$, there exists $T_3>0$ such that for every $t\geq T_3$, we obtain
\begin{eqnarray*}
||\int_t^{+\infty} B_{2,0}^{\theta_2}(t,s,y(s),x(s))ds\;|| &\leq&  \int_t^{+\infty}|| B_{2,0}^{\theta_2}(t,s,y(s),x(s))||ds \\
 &\leq &  \int_t^{+\infty}\theta_2(t,s)\widehat{B}_{2,0}(s,y(s),x(s))ds\\
 & < & \epsilon \int_t^{+\infty}\theta_2(t,s)ds\\
 & < & \epsilon P_2\; .
\end{eqnarray*} 
This proves that the operator $F_2$ leaves invariant the space $ AAA(\R^+; \X)$.
\end{proof}

At this point, we have given some preliminary results on almost and asymptotically almost automorphic functions needed for our approach. The results on existence and uniqueness for the equations of our interests are a combination of the previous one and of the following two important abstract theorems, which are applied to ensure the local existence and uniqueness of the solution to abstract equations.

\begin{thm}\label{teos1}
Let $(\X,||\cdot||)$ be a Banach space and $T:\X \to \X$ an operator, $\varrho>0$ a real number and $\Delta_0=\{y \in \X : \|y-y_0\|\leq \varrho\},$ where
$y_0=T(0)$, with $||y_0||\leq\varrho$. Suppose that $T$ satisfies $||Tx-Ty||\leq L_T||x-y||$, for
all $ \ x,y \ \in  \Delta_0,$ and the constants $L_T,\varrho,||y_0||_{\infty}$ satisfy
$L_T<\frac{\varrho}{\varrho+||y_0||}.$ Then, the equation $Tu=u$ has a unique solution in $\Delta_0$.
\end{thm}
Note that, in this theorem the closed ball $\Delta_0$ contains the zero vector of the Banach space $\X$. Since the Banach's contraction principle only need a complete metric space, it is also possible that the set $\Delta_0$ do not contain the zero vector. This situation is stated in the following theorem:

\begin{thm}\label{teos2}
Let $(\X,||\cdot||)$ be a Banach space and $T:\X \to \X$ be an  operator, $\varrho>0$ a real number and $\Delta_0=\{y \in \X : \|y-y_0\|\leq \varrho\}$,
where $y_0=T(0)$. Suppose that $T$ satisfies  $||Tx-Ty||\leq L_T||x-y||,$ for all $ \ x,y \ \in \Delta_0,$ with $L_T<1,$
and the constants $y_0,Ty_0,L_T,\varrho$ satisfy $0<\theta=(1-L_T)^{-1}||Ty_0-y_0||\leq\varrho.$ Then, the equation $Tu=u$  has a unique solution in $\Delta_0.$
\end{thm}
The previous two abstract local theorems have appeared in the work of M. Pinto, G. Robledo \& V. Torres (see \cite{MPGRVT}), in which the authors give its short proofs; they also consider applications (among others things) to the linear attractivity  in quasilinear difference systems. Here we do not reproduce their proofs, but we will give full details of the proof in the applications of them; that is, of the theorems on existence and uniqueness of solutions of the integral and integro-differential equations studied here.

\subsection{The almost periodic case}
Bochner almost periodic functions is a strong version than almost automorphic functions, if we compare its definitions, what wee see is that it demands uniform convergence in its definition instead of pointwise convergence; see comments after definition \ref{def1} and see also the references \cite{39Cord,TDBook,29GM}. Since uniform convergence is stronger than pointwise convergence, the conditions for ensure the invariance of the almost periodic functions and of the asymptotically almost periodic functions under the operators appearing in lemmas \ref{lem12}, \ref{lem13} and \ref{l23} (respectively) need to be modified, in fact they will be weakened. As a first stage,  a weak notion of Bi-almost periodicity must be given. One of the authors of this work (M. Pinto) has studied existence and uniqueness of the pesudo-almost periodic solution to the equation (\ref{eq5})  in his work \cite{Pint10}. He used the notion of Bohr almost periodicity in his research and also give us a notion of Bi-almost periodicity in the Bohr sense. We refer the reader to the work \cite{Pint10} in order to analyze the conditions in the framework of (Bohr) almost periodicity; as a byproduct, the reader will clarify the differences between almost periodic and almost automorphic dynamics.
\section{Almost automorphic solutions of integral equations.}\label{aasol}

In this section we study the existence and uniqueness of the  almost automorphic solution for the integral equations  (\ref{eq5}) and (\ref{eq4}). Remember that the integral equation (\ref{eq5}), is:

$$
y(t)=f(t,y(t),y(a_0(t)))+\int_{-\infty}^{t}C_1(t,s,y(s),y(a_1(s)))ds+\int_{t}^{+\infty}C_2(t,s,y(s),y(a_2(s)))ds\; .
$$
%
%

\begin{thm}\label{th24}
Suppose that condition {\bf (H1)} holds, $C_1, C_2:\mathbb{R}\times \mathbb{R} \times \X \times \X \to \X$ are Bi-almost automorphic functions that satisfies conditions {\bf (H3)} and {\bf (H4)}. 
%
%
Let $f \in AA(\R\times \X \times \X; \X)$, \, $\varrho >0 $ and the set
$$\Delta_0=\{y \in AA(\R,\X):||y-y_0||_{\infty}\leq \varrho\}\, ,$$
with $y_0:\mathbb{R}\to \X$ defined by:
$$y_0(t)=f(t,0,0)+\int_{-\infty}^{t}C_1(t,s,0,0)ds+\int_{t}^{+\infty}C_2(t,s,0,0)ds\; .$$
Additionally that $||y_0||_{\infty}\leq\varrho$ and the following properties holds:\\
\begin{enumerate}
\item There exists a positive constant $L_f$  such that for all $ t \in \R, \ x,y,x_1,y_1 \in \Delta_0$:
$$||f(t,x(t),y(t))-f(t,x_1(t),y_1(t))||\leq L_f (||x(t)-x_1(t)||+||y(t)-y_1(t)||)\; .$$
\item The constants $\varrho,L_f,N_1,N_2$ 
satisfies the following inequality:
\begin{equation} \label{eq27}
2(L_f+N_1+N_2)< \dfrac{\varrho}{\varrho+||y_0||_{\infty}}. 
\end{equation}
\end{enumerate}
Then the integral equation (\ref{eq5}) has a unique almost automorphic solution in $\Delta_0$.
\end{thm}
\begin{proof}
Let us consider the operator $ \Gamma: AA(\R;\X) \to AA(\R;\X)$  defined  by 

\begin{equation}\label{OP01}
\Gamma y(t)= f(t,y(t),y(a_0(t)))+\int_{-\infty}^{t}C_1(t,s,y(s),y(a_1(s)))ds+\int_{t}^{+\infty}C_2(t,s,y(s),y(a_2(s)))ds\; .
\end{equation}
Since $AA(\R;\X)$ is a vector space and because of lemma \ref{l22} and corollary \ref{CorAA} we conclude that  $\Gamma $ is a well defined operator. Let us prove that $\Gamma (\Delta_0)\subseteq (\Delta_0)$; in fact, let $y\in \Delta_0$, then:
\begin{eqnarray*}
||\Gamma y(t)-y_0(t)||&\leq& ||f(t,y(t),y(a_0(t)))-f(t,0,0)||\\
&+&\int_{-\infty}^{t}||C_1(t,s,y(s),y(a_1(s)))-C_1(t,s,0,0)||ds\\
& +&\int_{t}^{+\infty}||C_2(t,s,y(s),y(a_2(s)))-C_2(t,s,0,0)||ds\\
&\leq& L_f(||y(t)||+||y(a_0(t))||)+ \int_{-\infty}^{t}\mu_1(t,s)(||y(s)||+||y(a_1(s))||)ds\\
&+& \int_{t}^{+\infty}\mu_2(t,s)(||y(s)||+||y(a_2(s))||)ds\\
&\leq& \left(  L_f+ \int_{-\infty}^{t}\mu_1(t,s) ds+ \int_{t}^{+\infty}\mu_2(t,s)ds \right) 2||y||_{\infty}\\
&\leq & \left(  L_f+N_1+ N_2 \right) 2||y||_{\infty}\\
& < & \varrho.
\end{eqnarray*}
Now, let us consider $y_1,y_2 \in \Delta_0$, then:
\begin{eqnarray*}
||\Gamma y_1(t)-\Gamma
y_2(t)||&\leq&||f(t,y_1(t),y_1(a_0(t)))-f(t,y_2(t),y_2(a_0(t)))||\\
&+&\int_{-\infty}^{t}||C_1(t,s,y_1(s),y_1(a_1(s)))-C_1(t,s,y_2(s),y_2(a_1(s)))||ds\\
& +&\int_{t}^{+\infty}||C_2(t,s,y_1(s),y_1(a_2(s)))-C_2(t,s,y_2(s),y_2(a_2(s)))||ds\\
&\leq & L_f(||y_1(t)-y_2(t)||+||y_1(a_0(t))-y_2(a_0(t))||)+\\
&+& \int_{-\infty}^{t}\mu_1(t,s) \Big{( } ||y_1(s)-y_2(s)||+||y_1(a_1(s))-y_2(a_1(s)) \Big{)}ds \\
&+& \int_{t}^{+\infty}\mu_2(t,s) \Big{( } ||y_1(s)-y_2(s)||+||y_1(a_2(s))-y_2(a_2(s)) \Big{)}ds\\
&\leq& 2 (L_f+N_1+N_2)||y_1-y_2||_{\infty}\; ,
\end{eqnarray*}
and by the inequality (\ref{eq27}) we have that $\Gamma$ is contractive. Therefore, the Banach fixed point Theorem give us the existence and uniqueness of the solution.

\end{proof}
As a consequence of the previous theorem, we have the following
\begin{cor}\label{th23}
Suppose that condition {\bf (H1)} holds and $C_1:\mathbb{R}\times \mathbb{R} \times \X \times \X \to \X$ is a Bi-almost automorphic function that satisfy conditions {\bf (H3)} and {\bf (H4)}. Let $f\in AA(\R\times\X\times\X;\X)$,  $\varrho>0$ and 
$$\Delta_0=\{y \in AA(\R,\X):||y-y_0||_{\infty}\leq \varrho\},$$
 where:
$$y_0(t)= f(t,0,0)+\int_{-\infty}^{t}C_1(t,s,0,0)ds\; .$$
Additionally, $||y_0||_{\infty}\leq\varrho$ and the following conditions holds

\begin{enumerate}
\item There exists a positive constant $L_f$ such that for all $t \in \R$ and $x,y,x_1,y_1 \in \Delta_0$:
 $$||f(t,x(t),y(t))-f(t,x_1(t),y_1(t))||\leq L_f \left( ||x(t)-x_1(t)||+||y(t)-y_1(t)|| \right)\; .$$
\item The constants $\varrho,L_f,N_1$ satisfy the inequality
    $$2(L_f+N_1)<\dfrac{\varrho}{\varrho + ||y_0||_{\infty}}.$$
\end{enumerate}
Then the integral equation (\ref{eq4})  has a unique almost automorphic solution in $\Delta_0$.
\end{cor}

%
%
%


Now we improve results in the spirit of theorem \ref{teos2} to the equations (\ref{eq5}) and (\ref{eq4}).

\begin{thm}Suppose that condition {\bf (H1)} holds, $C_1, C_2:\mathbb{R}\times \mathbb{R} \times \X \times \X \to \X$ are Bi-almost automorphic functions  that satisfy condition {\bf (H3)} and {\bf (H4)}. Let  $\Gamma$ be the operator defined in (\ref{OP01}), $f\in AA(\R\times\X\times\X;\X)$, $\rho>0$ and $\Delta_0=\{y\in AA(\R;\X): ||y-y_0||_{\infty}\leq\rho\}\; $, where 
$$y_0(t)=f(t,0,0)+\int_{-\infty}^{t}C_1(t,s,0,0)ds + \int_{t}^{+\infty}C_2(t,s,0,0)ds\; .$$
Also, for all $t \in \R$ and $x,y,x_1,y_1 \in \Delta_0$ we have:
$$||f(t,x(t),y(t))-f(t,x_1(t),y_1(t))||\leq L_f \left( ||x(t)-x_1(t)||+||y(t)-y_1(t)|| \right)\; .$$
Suppose that we have the inequality:
$$0<\theta=\left(1-2(L_f+N_1+N_2)\right )^{-1}||\Gamma y_0-y_0||_{\infty}\leq\rho\; ,$$
with $2(L_f+N_1+N_2)<1.$ Then the equation (\ref{eq5}) has a unique solution
$y\in AA(\R;\X)$ such that $||y-y_0||_{\infty}\leq \rho.$
\end{thm}

\begin{proof}
Let us consider the closed ball
$$\mathfrak{B}=\mathfrak{B}(y_0,\theta)=\{z\in AA(\R;\X): ||z-y_0||\leq\theta\}\, .$$
Let $z\in \mathfrak{B}(y_0,\theta)$, then we have:
\begin{eqnarray*}
||(\Gamma z)(t)-y_0(t)||&\leq&||(\Gamma z)(t)-(\Gamma y_0)(t)||+||(\Gamma y_0)(t)-y_0(t)||\\
&\leq&||f(t,z(t),z(a_0(t)))-f(t,y_0(t),y_0(a_0(t)))||+\\
&+&\int_{-\infty}^{t}||C_1(t,s,z(s),z(a_1(s)))-C_1(t,s,y_0(s),y_0(a_1(s)))||ds\\
& +&\int_{t}^{+\infty}||C_2(t,s,z(s),z(a_2(s)))-C_2(t,s,y_0(s),y_0(a_2(s)))||ds\\
& + & ||(\Gamma y_0)(t)-y_0(t)||\\
&\leq & L_f(||z(s)-y_0(s)||+||z(a_0(s))-y_0(a_0(s))||)+\\
&+& \int_{-\infty}^{t}\mu_1(t,s) \Big{( } ||z(s)-y_0(s)||+||z(a_1(s))-y_0(a_1(s))|| \Big{)}ds \\
&+& \int_{t}^{+\infty}\mu_2(t,s) \Big{( } ||z(s)-y_0(s)||+||z(a_2(s))-y_0(a_2(s))|| \Big{)}ds\\
&\leq& 2 (L_f+N_1+N_2)||z-y_0||_{\infty}+ ||\Gamma y_0-y_0||_{\infty}\\
&\leq&\theta.
\end{eqnarray*}
This computation means that $\Gamma (\mathfrak{B})\subseteq \mathfrak{B}$. On the other hand, we have that the desired Lipschitz's constant is 
$ L_{\Gamma}=2 (L_f+N_1+N_2)$ .
\end{proof}
In case $C_2\equiv 0$, the operator $\Gamma$ defined in (\ref{OP01}), becomes the following
\begin{equation}\label{OP02}
\Gamma y(t)= f(t,y(t),y(a_0(t)))+\int_{-\infty}^{t}C_1(t,s,y(s),y(a_1(s)))ds\; .
\end{equation}
In this way, we obtain the corollary that follows
\begin{cor}Suppose that condition {\bf (H1)} holds and $C_1:\mathbb{R}\times \mathbb{R} \times \X \times \X \to \X$ is a Bi-almost automorphic function which satisfy conditions {\bf (H3)} and {\bf (H4)}. Let $\Gamma$ be the operator defined in (\ref{OP02}), $f\in AA(\R\times\X\times\X;\X)$, $\rho>0$ and $\Delta_0=\{y\in AA(\R;\X): ||y-y_0||_{\infty}\leq\rho\}$, where 
$$y_0(t)=f(t,0,0)+\int_{-\infty}^{t}C_1(t,s,0,0)ds\; ;$$
also, for all $t \in \R$ and $x,y,x_1,y_1 \in \Delta_0$ we have:
$$||f(t,x(t),y(t))-f(t,x_1(t),y_1(t))||\leq L_f \left( ||x(t)-x_1(t)||+||y(t)-y_1(t)|| \right)\; .$$
Suppose that $0<\theta=(1-2(L_f+N_1))^{-1}||\Gamma y_0-y_0||_{\infty}\leq\rho,$
with $2(L_f+N_1)<1$, then the equation (\ref{eq4}) has a unique solution $y\in AA(\R;\X)$ such that
$||y-y_0||_{\infty}\leq \rho.$\\
\end{cor}

\noindent Let us assume the following conditions:\\

\begin{enumerate}
\item [$(K_1):$] $f \in AA(\R\times \X;\X)$ and there exist a continuous and bounded functions $L_f:
\R^+ \to \R^+$,  such that: for all $r>0$ and for all  $ x,y,x_1,y_1 \in B(0,r)=\{x\in \X: ||x||\leq r\}$ we have:
    $$||f(t,x,x_1)-f(t,y,y_1)||\leq L_f(r)\left( ||x-y||+||x_1-y_1|| \right)\; .$$

   
\noindent \item [$(K_2):$]  $\displaystyle\sup_{r>0}\left( r-2rL_f(r)- 2rN_1-2rN_2 \right) > \displaystyle\sup_{t \in \R}||f(t,0,0)||+\alpha_1 +\alpha_2\; .$
\end{enumerate}

With this new conditions, we have the following Theorem:
\begin{thm}\label{th25}
Suppose that condition {\bf (H1)} holds and $C_1,C_2:\mathbb{R}\times \mathbb{R} \times \X \times \X \to \X$ are Bi-almost automorphic function that satisfy conditions {\bf (H3)} and {\bf (H4)} and further that conditions $(K_1)$ and $(K_2)$ holds. Then the integral equation (\ref{eq5}) has a unique
almost automorphic solution.
\end{thm}
\paragraph{Proof.} 
\begin{proof}
Let us consider the operator $\Gamma : AA(\R;\X)\to AA(\R;\X)$ defined in (\ref{OP01}). Since $AA(\R;\X)$ is a vector space and because of lemma \ref{l22} and corollary \ref{CorAA} we conclude that  $\Gamma $ is a well defined operator.
The condition  $(H_2)$ implies the existence of a real number $R>0$ such that:
\begin{eqnarray}\label{eq28}
R-2RL_f(R)- 2RN_1-2RN_2  > \displaystyle\sup_{t \in \R}||f(t,0,0)||+\alpha_1 +\alpha_2
\end{eqnarray}
Let us consider the set $\Omega_0=\{y \in AA(\R;\X): ||y||_{\infty}\leq R\}.$
We need to prove that $\Gamma (\Omega_0) \subseteq \Omega_0$ and that $\Gamma$ is contractive, in fact:\\
1). Let $y \in \Omega_0$, then:
\begin{eqnarray*}
||\Gamma y(t)||&\leq&||f(t,y(t),y(a_0(t)))-f(t,0,0)||+||f(t,0,0)||+\\
&+&\int_{-\infty}^{t}||C_1(t,s,y(s),y(a_1(s)))-C_1(t,s,0,0)||ds+\int_{-\infty}^{t}||C_1(t,s,0,0)||ds+\\
& +&\int_{t}^{+\infty}||C_2(t,s,y(s),y(a_2(s)))-C_2(t,s,0,0)||ds+\int_{t}^{+\infty}||C_2(t,s,0,0)||]ds\\
&\leq& 2R L_f(R)+\sup_{s\in\R}||f(s,0,0)||+2R N_1+ 2RN_2+\alpha_1 +\alpha_2 \\
&\leq&R.
\end{eqnarray*}
The last inequality is justified by (\ref{eq28}).\\
2). Let $y_1,y_2 \in \Omega_0 $, then:
\begin{eqnarray*}
||\Gamma y_1(t)-\Gamma
y_2(t)||&\leq&||f(t,y_1(t),y_1(a_0(t)))-f(t,y_2(t),y_2(a_0(t)))||+\\
&+&\int_{-\infty}^{t}||C_1(t,s,y_1(s),y_1(a_1(s)))-C_1(t,s,y_2(s),y_2(a_1(s)))||ds+\\
& +&\int_{t}^{+\infty}||C_2(t,s,y_1(s),y_1(a_2(s)))-C_2(t,s,y_2(s),y_2(a_2(s)))||ds\\
&\leq& \big{(} 2L_f(R)+ 2N_1+2N_2 \big{)}||y_1-y_2||_{\infty}.
\end{eqnarray*}
From the inequality (\ref{eq28}) we have
$$R-2RL_f(R)-2RN_1-2RN_2>0,$$
then:
$$1>2L_f(R)+ 2N_1+2N_2.$$
Therefore $\Gamma$ is a contractive operator, consequently has a unique fixed point.
\end{proof}
\begin{cor}
Suppose that condition {\bf (H1)} holds and $C_1,C_2:\mathbb{R}\times \mathbb{R} \times \X \times \X \to \X$ are Bi-almost automorphic function that satisfy conditions {\bf (H3)} and {\bf (H4)} and further that conditions $(K_1)$ holds with $L_f(\cdot)=L_f$ a positive constant. Moreover, the inequality 
$$2 (L_f+N_1+N_2)<1\; ,$$ 
holds. Then the integral equation (\ref{eq5}) has a unique almost automorphic solution.
\end{cor}
\begin{proof}
Consider the operator $\Gamma : AA(\R;\X)\to AA(\R;\X)$ defined in (\ref{OP01}). Since $AA(\R;\X)$ is a vector space and because of lemma \ref{l22} and corollary \ref{CorAA} we conclude that  $\Gamma $ is a well defined operator. Moreover, because every almost automorphic function is bounded and also by hypothesis, we can conclude the following inequality:
$$\alpha_1+\alpha_2+\sup_{s\in\R}||f(s,0,0)||<+\infty\; .$$
Also by hypothesis we have:
$$1-2(L_f+N_1+N_2)>0\, .$$
Therefore, there exist a real number $R_0>0$ such that for all $R\geq R_0$ we have
\begin{eqnarray*}
R-2R(L_f+N_1+N_2) > \sup_{s\in\R}||f(s,0,0)||+\alpha_1+\alpha_2\; .
\end{eqnarray*}
Now, it is clear that the conclusion follows from theorem \ref{th25}.
\end{proof}
When we have the most simple equation (\ref{eq4}), it is possible to give the precise conditions in order to deduce the existence of its unique almost automorphic solution. For that, let us consider the following conditions:\\
\begin{itemize}
\item [$(k_1):$] $f \in AA(\R \times \X,\X)$ and there exists $L_f:\R^+ \to \R^+$ a continuous and bounded functions such that for all  $ r >0$ and for all $ x,y,x_1,y_1 \in  B(0,r)=\{x\in \X: ||x||\leq r\} $ we have:
    $$||f(t,x,x_1)-f(t,y,y_1)||< L_f(r)\left( ||x-y||+||x_1-y_1|| \right).$$ 
\item [$(k_2):$] $\displaystyle\sup_{r>0}\left(r-2rL_f(r)-2rN_1\right)>
\alpha_1+\sup_{s\in\R}||f(s,0,0)||.$
\end{itemize}
 
\begin{thm}\label{lk} 
Suppose that condition {\bf (H1)} holds and $C_1:\mathbb{R}\times \mathbb{R} \times \X \times \X \to \X$ is a Bi-almost automorphic function that satisfy conditions {\bf (H3)} and {\bf (H4)} and further that conditions $(k_1)$ and $(k_2)$ holds. Then, the integral equation (\ref{eq4}) has a unique almost automorphic solution.
\end{thm}

\begin{cor} Suppose that condition {\bf (H1)} holds and $C_1:\mathbb{R}\times \mathbb{R} \times \X \times \X \to \X$ is a Bi-almost automorphic function that satisfy conditions {\bf (H3)} and {\bf (H4)} and further that condition $(k_1)$ holds  with $L_f(\cdot)=L_f$ a positive constant. If the inequality 
$$2(L_f+N_1)<1$$ 
hold, then the equation (\ref{eq4}) has a unique almost automorphic solution.
\end{cor}


\section{A Bohr-Neugebauer type result.}\label{secB-N}
As it is well known Massera's theorem for periodic linear systems asserts that: a periodic linear system has a  periodic solution if and only if it has a bounded solution. The Bohr-Neugebauer result, is an extension of the periodic setting to the almost periodic one and it asserts that a bounded solution of an almost periodic linear system is actually almost periodic, and obviously all almost periodic solutions are bounded. In this section we are able to give a result in the direction of Bohr-Neugebauer for integral equations; i.e. we ensure that a solution with relatively compact range (and thus bounded) of the integral equations of our interest is actually almost automorphic.\\

Let $||\cdot||_{es}$ be the essential supremum norm, 
  we need the following lemma whose proof is immediate:\\

\begin{lem}\label{AUXLEM1}Let the positive functions $a,\lambda ,v:\R \to \R^+,\, D:\R\times\R\to\R^+$ with $v$ essentially bounded and 
$$\sup_{t \in \R}\Big{(}\int_{-\infty}^{t}D_1(t,s)\lambda_1(s)ds+\int_{t}^{+\infty}D_2(t,s)\lambda_2(s)ds\Big{)}=\varrho < 1\; .$$ 
If $v$ satisfy the integral inequality:
$$v(t) \leq a(t)+\int_{-\infty}^{t}D_1(t,s)\lambda_1(s)v(s)ds+\int_{t}^{+\infty}D_2(t,s)\lambda_2(s)v(s)ds, \ t\in \R\, ,$$
then \,  $||v||_{es} \leq \dfrac{||a||_{\infty}}{1-\varrho}.$ 
\end{lem}

\begin{rmk}
We remark that if $v$ is bounded, then the conclusion of the previous lemma is $||v||_{\infty} \leq \dfrac{||a||_{\infty}}{1-\varrho},$ which implies that $v(t) \leq \dfrac{||a||_{\infty}}{1-\varrho},$ for all $t \in \R$. Therefore, if the function $a$ is the zero constant, then the function $v$ is also the zero constant.
\end{rmk}
 
%
%

\begin{thm}\label{Massteo}
Suppose that $C_1, C_2:\mathbb{R}\times \mathbb{R} \times \X \times \X \to \X$ are Bi-almost automorphic functions that satisfies conditions {\bf (H3)} and {\bf (H4)}, and let $f \in AA(\R\times \X \times \X; \X)$ be a $L_f$-Lipschitz, that is:
$$||f(t,x,y)-f(t,x_1,y_1)||\leq L_f \left(||x-x_1||+||y-y_1|| \right)\, ,$$
for $(x,y),\, (x_1,y_1)$ on bounded subsets of $\X\times \X$. Further, suppose that 
$$L_f+\sup_{t \in \R}\left( \int_{-\infty}^t \mu_1(t,s)ds+\int_t^{+\infty}\mu_2(t,s)ds\right)=\rho <1\, ,$$
and for $i=0,1,2$ the functions $a_i \in AA(\R;\X)$. Then, a bounded continuous solution of the integral equation (\ref{eq4}) is almost automorphic if and only if it has relatively compact range.
\end{thm}

\begin{proof}
By the properties of almost automorphic functions, the necessity is immediate. Therefore, we prove the
sufficient condition.\\
Let $\{s_n^{''}\}$ be an arbitrary sequence in $\R$, since $f \in AA(\R\times\X \times\X;\X)$ there exist a subsequence
$\{s_n^{'}\}\subseteq \{s_n^{''}\}$ such that, if $\mathcal{B}$ is a bounded subset of $\X \times\X$, the following limits holds:
$$\lim_{n \to +\infty }f(t+s_n^{'},x,y)=:\tilde f(t,x,y), \lim_{n \to +\infty }\tilde f(t-s_n^{'},x,y)=f(t,x,y),  \ t \in \R,
\, \forall \; (x , y)  \in \mathcal{B}.$$
Moreover, for all $(x , y)  \in \mathcal{B}$ we have
$$\lim_{n\to \infty}C_1(t+s_n',s+s_n',x,y):=\tilde{C}_1(t,s,x,y)\; ,\; \; \lim_{n\to \infty}\tilde{C}_1(t-s_n',s-s_n',x,y)=C_1(t,s,x,y)\; .$$
and also
$$\lim_{n\to \infty}C_2(t+s_n',s+s_n',x,y):=\tilde{C}_2(t,s,x,y)\;, \; \; \lim_{n\to \infty}\tilde{C}_2(t-s_n',s-s_n',x,y)=C_2(t,s,x,y)\; .$$
Let $y\in BC(\mathbb{R}\; \X)$ be a solution of the integral equation (\ref{eq4}) with range, $\mathcal{R}_y$, relatively compact. Then, for $i=0,1,2$, the range of each functions $y\circ a_i:\mathbb{R}\to \X$ are relatively compact. Let $\overline{\mathcal{R}_y}=K$. Since  $\{y(t+s_n^{'})\}_{n\in \N}$ is a sequence in $K$,  by the equivalent properties of compactness in metric space, there exists a subsequence $\{s_n\}\subseteq \{s_n^{'}\}_{n\in \N}$ such that the sequence $\{y(t+s_n)\}_{n\in \N}$ converge to $\tilde{y}(t)\in \X$; this define a function $\tilde{y}:\mathbb{R}\to \X$. Similarly, we have that the sequences $\{y(a_i(t+s_n))\}_{n\in \N}$ converge to $\tilde{y}_i(t)$ and therefore define the functions $\tilde{y}_i:\R \to \X$. \\
Now, let us define the new function
$$\xi(t):=\tilde f(t,\tilde y(t), \tilde y_0(t))+\int_{-\infty}^{t}\tilde C_1(t,s,\tilde y(s),\tilde y_1(s))ds+\int_{t}^{+\infty}\tilde C_2(t,s,\tilde y(s),\tilde y_2(s))ds\, .$$
Then we have
\begin{eqnarray*}
||y(t+s_n)-\xi(t)|| &\leq & ||f(t+s_n,y(t+s_n),y\circ a_0(t+s_n))-f(t+s_n,\tilde{y}(t),\tilde{y}_0(t))||+\\
&+&||f(t+s_n,\tilde{y}(t),\tilde{y}_0(t))-\tilde{f}(t,\tilde{y}(t),\tilde{y}_0(t)) ||+\\
&+&\int_{-\infty}^{t}\Big{|}\Big{|} C_1(t+s_n,s+s_n, y(s+s_n),y\circ a_1(s+s_n))-\tilde C_1(t,s,\tilde y(s),\tilde y_1(s))\Big{|}\Big{|}ds +\\
&+& \int_{t}^{+\infty}\Big{|}\Big{|} C_2(t+s_n,s+s_n, y(s+s_n),y\circ a_2(s+s_n))-\tilde C_2(t,s,\tilde y(s),\tilde y_2(s))\Big{|}\Big{|}ds\, .
\end{eqnarray*}
Therefore using the Lebesgue's dominate convergence theorem, we have
$$\lim_{n\to \infty}||y(t+s_n)-\xi(t)||=0\, ,$$
which means that $\xi(t)=\tilde{y}(t)$. The previous calculus implies that the function $\tilde{y}$ satisfy the integral equation
$$\tilde{y}(t)=\tilde f(t,\tilde y(t), \tilde y_0(t))+\int_{-\infty}^{t}\tilde C_1(t,s,\tilde y(s),\tilde y_1(s))ds+\int_{t}^{+\infty}\tilde C_2(t,s,\tilde y(s),\tilde y_2(s))ds\, .$$
Let us prove that $\tilde{y}(t-s_n)\to y(t)$ as $n \to \infty$. In fact, first observe that
\begin{eqnarray*}
||\tilde{y}(t-s_n)-y(t)|| &\leq & ||\tilde{f}(t-s_n,\tilde{y}(t-s_n),\tilde{y}_0(t-s_n))-f(t,\tilde{y}(t-s_n),\tilde{y}_0(t-s_n))||+\\
&+&||f(t,\tilde{y}(t-s_n),\tilde{y}_0(t-s_n))-f(t,y(t),y\circ a_0(t)) ||+\\
&+&\int_{-\infty}^{t}\Big{|}\Big{|}\tilde C_1(t-s_n,s-s_n,\tilde y(s-s_n),\tilde y_1(s-s_n))- C_1(t,s,y(s),y\circ a_1(s))\Big{|}\Big{|}ds +\\
&+& \int_{t}^{+\infty}\Big{|}\Big{|}\tilde C_2(t-s_n,s-s_n,\tilde y(s-s_n),\tilde y_2(s-s_n))- C_2(t,s,y(s),y\circ a_2(s))\Big{|}\Big{|}ds\\
&\leq & \Lambda_{0,n}(t)+L_f  \Big{(}  ||\tilde y(t-s_n)-y(t)||+||\tilde y_0(t-s_n)-y\circ a_0 (t)||  \Big{)} +\\
&+& \Lambda_{1,n}(t)+\int_{-\infty}^t  \mu_1(t,s)  \Big{(} ||\tilde y(s-s_n)-y(s)|| +||\tilde{y}_1(s-s_n)-y\circ a_1(s)|| \Big{)} ds+\\
&+& \Lambda_{2,n}(t)+\int_{t}^{+\infty}  \mu_2(t,s)  \Big{(} ||\tilde y(s-s_n)-y(s)|| +||\tilde{y}_2(s-s_n)-y\circ a_2(s)|| \Big{)} ds\\
&\leq & \sum_{i=0}^3\Lambda_{i,n}+ L_f  ||\tilde y(t-s_n)-y(t)||+ \int_{-\infty}^t  \mu_1(t,s) ||\tilde y(s-s_n)-y(s)|| ds\, +\\
&+& \int_{t}^{+\infty}  \mu_2(t,s)||\tilde y(s-s_n)-y(s)||ds .\\
\end{eqnarray*}
Where,
\begin{eqnarray*}
\Lambda_{0,n}(t)&:=&||\tilde{f}(t-s_n,\tilde{y}(t-s_n),\tilde{y}_0(t-s_n))-f(t,\tilde{y}(t-s_n),\tilde{y}_0(t-s_n))||\, ,\\
\Lambda_{1,n}(t)&:=& \int_{-\infty}^{t}\Big{|}\Big{|}\tilde C_1(t-s_n,s-s_n,\tilde y(s-s_n),\tilde y_1(s-s_n))- C_1(t,s,\tilde y(s-s_n),\tilde y_1(s-s_n))\Big{|}\Big{|}ds\, ,\\
\Lambda_{2,n}(t)&:=&\int_{t}^{+\infty}\Big{|}\Big{|}\tilde C_2(t-s_n,s-s_n,\tilde y(s-s_n),\tilde y_2(s-s_n))- C_2(t,s,\tilde y(s-s_n),\tilde y_2(s-s_n))\Big{|}\Big{|}ds\, ,\\
\Lambda_{3,n}(t)&:=&L_f ||\tilde{y}_0(t-s_n)-y\circ a_0(t)||+ \int_{-\infty}^t  \mu_1(t,s) ||\tilde{y}_1(s-s_n)-y\circ a_1(s)||ds\, + \\
&+& \int_{t}^{+\infty}  \mu_2(t,s) ||\tilde{y}_2(s-s_n)-y\circ a_2(s)||ds\; .
\end{eqnarray*}
It is not difficult to check that for $i=0,1,2$ we obtain $\Lambda_{i,0}(t)\to 0$, if $n \to \infty$.
Furthermore, since $a_i\in  AA(\R ; \X)$ then we also have $\Lambda_{3,n}(t)\to 0$, if $n \to \infty$.

On the other hand, since $||\tilde{y}(t-s_n)-y(t)||\leq 2||y||_{\infty}$  and $y$ is bounded, then by Bolzano-Weierstrass's theorem there exist a subsequence $\{s_{\tau}\}\subseteq \{s_n\}$ such that $||\tilde y(t-s_\tau)-y(t)||\to \eta(t)$
pointwise when $\tau \to +\infty$. Therefore, we have
$$\eta(t)\leq L_f \eta(t) + \int_{-\infty}^t\mu_{1}(t,s)\eta(s)ds + \int_t^{+\infty}\mu_{2}(t,s)\eta(s)ds$$
%
Now, using lemma \ref{AUXLEM1} we obtain that $\eta$ is the null function, from where we immediately deduce that  $||\tilde y(t-s_n)-y(t)||\to 0$, as $n \to \infty$.
\end{proof}

The corresponding result for the integral equation (\ref{eq4}) can be easily deduced from the previous theorem, but we prefer to state it separately as a matter of completeness.
\begin{thm}\label{Massteo2}
Suppose that $C_1:\mathbb{R}\times \mathbb{R} \times \X \times \X \to \X$ is a Bi-almost automorphic function that satisfy  conditions {\bf (H3)} and {\bf (H4)}, and let $f \in AA(\R\times \X \times \X; \X)$ be $L_f$-Lipschitz, that is:
$$||f(t,x,y)-f(t,x_1,y_1)||\leq L_f \left(||x-x_1||+||y-y_1|| \right)\, ,$$
for $(x,y),\, (x_1,y_1)$ on bounded subsets of $\X\times \X$. Further suppose that 
$$L_f+\sup_{t \in \R}  \int_{-\infty}^t \mu_1(t,s)ds  =\rho <1\, ,$$
and for $i=0,1$ the functions $h_i \in AA(\R;\X)$. Then, a bounded solution of the integral equation (\ref{eq4}) is almost automorphic if and only if it has relatively compact range.
\end{thm}

\begin{rmk} We remark that, if $\X$ is a finite dimensional Banach space, then the conclusion of theorem \ref{Massteo} (or theorem \ref{Massteo2}) says that:  {\sl a continuous solution of the integral equation (\ref{eq5}) (or of (\ref{eq4})) is almost automorphic  if and only if it is bounded}.
\end{rmk}
It seems very interesting to analyze the existence of solutions, to the integral equations treated here, inside the class of functions which has relatively compact range, this new class certainly generalizes the periodic, almost periodic and the almost automorphic ones, see \cite{Kloeden}. We will back to this problem in a future work.

\section{Asymptoticaly almost automorphic solution of integro-differential equations.}\label{aaasol}
In this section, we analyze the existence and uniqueness of the asymptotically almost automorphic solution of the equations (\ref{eqNew}) and (\ref{eqX})-(\ref{eqXX}). Let us remember from hypothesis {\bf (H5)} that for $i=1,2$; $B_i$ has the decomposition $B_i=B_i^a+B_{i,0}^{{\theta}_i}$. Before to present the results for asymptotically almost automorphic solution, let us present the following condition:\\
 
\noindent {\bf (H6)}. The functions $B_{i,0}^{{\theta}_i}$ are $\mu_3^i$-Lipschitz; that is, there exist positive functions $\mu_3^i:\R \times \R \to \R^+$ such that 
$$||B_{i,0}^{{\theta}_i}(t,s,x,y)-B_{i,0}^{{\theta}_i}(t,s,x_1,y_1)||\leq \mu_3^i(t,s)\left( ||x-x_1||+||y-y_1||\right)\, .$$
Moreover
$$\sup_{t \geq 0}\left( \int_{0}^t\mu_3^1(t,s)ds +\int_{t}^{+\infty}\mu_3^2(t,s)ds \right)=Q_1<+\infty\, .$$

\subsection{Asymptotically almost automorphic solution to the integral equation (\ref{eqNew}).\\}
Let us start with the following application of theorem \ref{teos1}:

\begin{thm}\label{thAAA24}
Suppose that condition {\bf (H2)} holds, $B_1, B_2:\mathbb{R}\times \mathbb{R} \times \X \times \X \to \X$ are functions that satisfies condition {\bf (H5)} and {\bf (H6)}.
%
%
Let $f \in AAA(\R\times \X \times \X; \X)$, \, $\varrho >0 $ and the set
$$\Delta_0=\{y \in AAA(\R,\X):||y-y_0||_{\infty}\leq \varrho\}\, ,$$
with $y_0:\mathbb{R}\to \X$ defined by:
$$y_0(t)=f(t,0,0)+\int_{0}^{t}B_1(t,s,0,0)ds+\int_{t}^{+\infty}B_2(t,s,0,0)ds\; .$$
Additionally that $||y_0||_{\infty}\leq\varrho$ and the following properties holds:\\
\begin{enumerate}
\item There exists a positive constant $L_f$  such that for all $ t \in \R, \ x,y,x_1,y_1 \in \Delta_0$:
$$||f(t,x(t),y(t))-f(t,x_1(t),y_1(t))||\leq L_f (||x(t)-x_1(t)||+||y(t)-y_1(t)||)\; .$$
\item The constants $\varrho,L_f,N_1,N_2$ 
satisfies the following inequality:
\begin{equation} \label{eq27}
2(L_f+Q_1)< \dfrac{\varrho}{\varrho+||y_0||_{\infty}}. 
\end{equation}
\end{enumerate}
Then the integral equation (\ref{eqNew}) has a unique asymptotically almost automorphic solution in $\Delta_0$.
\end{thm}
As is natural, in the proof of theorem \ref{thAAA24}, it is important the operator $\Pi: AAA(\R^+;\X) \to AAA(\R^+;\X)$ such that:

\begin{equation}\label{Opaaa1}
\Pi y(t)= f(t,y(t),y(b_0(t)))+\int_0^t B_1(t,s,y(s),y(b_1(s)))ds +
\int_t^{+\infty}B_2(t,s,y(s),y(b_2(s)))ds\; .
\end{equation}
The application of theorem \ref{teos2} is the following  result
\begin{thm}Suppose that condition {\bf (H2)} holds, $B_1, B_2:\mathbb{R}\times \mathbb{R} \times \X \times \X \to \X$ are Bi-almost automorphic functions that satisfy condition {\bf (H5)} and {\bf (H6)}. Let  $\Pi$ be the operator defined in (\ref{Opaaa1}), $f\in AAA(\R\times\X\times\X;\X)$, $\rho>0$ and $\Delta_0=\{y\in AAA(\R;\X): ||y-y_0||_{\infty}\leq\rho\}\; $, where 
$$y_0(t)=f(t,0,0)+\int_{0}^{t}B_1(t,s,0,0)ds + \int_{t}^{+\infty}B_2(t,s,0,0)ds\; .$$
Also, for all $t \in \R$ and $x,y,x_1,y_1 \in \Delta_0$ we have:
$$||f(t,x(t),y(t))-f(t,x_1(t),y_1(t))||\leq L_f \left( ||x(t)-x_1(t)||+||y(t)-y_1(t)|| \right)\; .$$
Suppose that we have the inequality:
$$0<\theta=\left(1-2(L_f+Q_1)\right )^{-1}||\Pi y_0-y_0||_{\infty}\leq\rho\; ,$$
with $2(L_f+Q_1)<1.$ Then the equation (\ref{eqNew}) has a unique solution
$y\in AAA(\R;\X)$ such that $||y-y_0||_{\infty}\leq \rho.$
\end{thm}
In order to improve another result, let us note that if condition {\bf (H2)}, {\bf (H5)} and {\bf (H6)} holds, then the functions $t\to \int_0^t B_1(t,s,0,0)ds$ and $t \to \int_{t}^{+\infty} B_2(t,s,0,0)ds$  are bounded; thus, we can define the following real numbers
 $$\sup_{t\geq 0}||\int_0^t B_1(t,s,0,0)ds||=:\gamma_1<+\infty\, ;$$
 $$\sup_{t\geq 0}||\int_{t}^{+\infty} B_2(t,s,0,0)ds||=:\gamma_2<+\infty\; .$$ 

\noindent Let us introduce the following conditions:\\

\begin{enumerate}
\item [$(T_1):$] $f \in AAA(\R^+\times \X;\X)$ and there exist a continuous and bounded functions $L_f:
\R^+ \to \R^+$,  such that: for all $r>0$ and for all  $ x,y,x_1,y_1 \in B(0,r)=\{x\in \X: ||x||\leq r\}$ we have:
    $$||f(t,x,x_1)-f(t,y,y_1)||\leq L_f(r)\left( ||x-y||+||x_1-y_1|| \right)\; .$$


\noindent \item [$(T_2):$]  $\displaystyle\sup_{r>0}\left( r-2rL_f(r)- 2rQ_1 \right) > \displaystyle\sup_{t \geq 0}||f(t,0,0)||+\gamma_1 +\gamma_2\; .$

\end{enumerate}
Now we have the following statement:

\begin{thm}\label{th25}
Suppose that condition {\bf (H2)} holds and $B_1,B_2:\mathbb{R}\times \mathbb{R} \times \X \times \X \to \X$ are functions  that satisfy conditions {\bf (H5)} and {\bf (H6)}, additionally that conditions $(T_1)$ and $(T_2)$ holds. Then, the integral equation (\ref{eqNew}) has a unique asymptotically almost automorphic solution.
\end{thm}
We comment that, similarly to equation (\ref{eq5}), it is possible to  give a Bohr-Neugebauer's type result in the context of asymptotically almost automorphic solutions to the equation (\ref{eqNew}).

\subsection{Asymptotically almost automorphic mild solution to the integro-differential equation (\ref{eqX})-(\ref{eqXX}).\\}\label{SecAT}

In this subsection, we study the existence of a unique asymptotically almost automorphic mild solution of the non-local integro-differential equation (\ref{eqX})-(\ref{eqXX}). Let us rewrite the integro-differential equation in the following convenient way, in which appears a causal operator:

\begin{eqnarray}
u'(t)&=& A(t)u(t)+F(t,u(t),\mathcal{B}u(t)),\ t\geq0\; , \label{eqX01}\\
u(0)&= & u_0 + h(u)\, . \label{eqXX01}
\end{eqnarray}
In which $$F(t,u(t),\mathcal{B}u(t)):= \mathcal{B}u(t) + g(t,u(t))\, ,$$ 
and 
$$\mathcal{B}u(t):=\int_0^t{B(t,s)u(s)ds}\, .$$
%
%
where $u_0 \in \X$, $A(t):D(A(t))\subset \X\rightarrow \X , \, \; \, t\in \R^+$ and $B(t,s):D(B(t,s))\subset \X\rightarrow \X ,\ t\geq s\geq  0$ are linear operators on the Banach space $\X$; and $g(\cdot,\cdot)$  is an asymptotically almost automorphic function.


 Let $J:=\R^+$ or $\R$, $\mathbb{X}$ a real or complex Banach space and $\mathcal{B}(\mathbb{X})$ the Banach algebra of all bounded linear operators from $\mathbb{X}$ to itself. Define the following set
$$\triangle_{J}^+:=\{(t,s)\in J \times J\; :\; t\geq s \}\; .$$

\begin{defn}
The operator valued function $U: \triangle_{J}^+ \to \mathcal{B}(\mathbb{\X})$ is say to be an {\bf evolution operator} on $J$, if it meet the following conditions:
\begin{enumerate}
\item For every $t \in J$: $U(t,t)=I$ ;
\item For all $t\geq s \geq r$ in $J$: $U(t,s)U(s,r)=U(t,r)$ ;
\item $U$ is strongly continuous.
\end{enumerate}
\end{defn}

The collection 
$$\mathcal{U}=\{ U: \triangle_{J}^+ \to \mathcal{B}(\mathbb{\X})\; : \; U {\rm \, is \, an\, evolution\, operator } \},$$
is called an {\bf evolution family}.


The following technical assumptions are needed in the study of equation (\ref{eqX01})-(\ref{eqXX01}). Because of the nature o the equation, the assumptions are posed on the semi axis $\mathbb{R}^+$; but, it also can be given in the real line (see for instance \cite{ZCWLin,HSDWLGMN,ManSch}), we will need this case in the final section of this work.

\begin{enumerate}
\item [{\bf (AT)}] The Acquistapace-Terrini conditions: there exists constants $\lambda_0 \geq 0,\,  \theta \in (\frac{\pi}{2},\pi), K_1, K_2 \geq 0$ and $\beta_1,\beta_2 \in (0,1]$ with $\beta_1 + \beta_2 >1$ such that for $t,s \in \mathbb{R}^+$ and $\lambda \in \Sigma_{\theta}=\{\lambda \in \mathbb{C}-\{0\}\, : \, |arg(\lambda)|\leq \theta\}$ \, we have

$$\Sigma_\theta \cup \{0\} \subset \rho (A(t)-\lambda_0), \; \; ||R(\lambda,A(t)-\lambda_0)||\leq \dfrac{K_1}{1+|\lambda|}\; \,$$
and
$$||[A(t)-\lambda_0]R(\lambda,A(t)-\lambda_0)[R(\lambda_0,A(t))-R(\lambda_0,A(s))]||\leq K_2|t-s|^{\beta_1}|\lambda|^{-\beta_2}\, .$$

\end{enumerate}
It is well know that, under this condition there exists a unique evolution family $\mathcal{U}$ 
which governs the linear equation $$x'(t)=A(t)x(t)\, ,$$
see \cite{Acq,AcqTerr} for original references.\\

 The Acquistapace-Terrini conditions have been used extensively in the study of existence of mild solutions to several nonautonomous evolution equations, see for instance the cited references \cite{ZCWLin,HSDWLGMN,ManSch} and other references therein. \\

\noindent {\bf (bAA)} The application $(t,s)\to U(t,s)x$ is Bi-almost automorphic uniformly for all $x$ on bounded subset of $\mathbb{X}$.\\

\noindent {\bf (Ex)} The evolution family $U(t,s)$ is exponentially stable; That is, there exists $M>0$ and $\delta >0$ such that $||U(t,s)||\leq M e^{-\delta (t-s)}$, for $t \geq s$.\\

\begin{defn}
A mild solution of the equation (\ref{eqX01})- (\ref{eqXX01}), is a continuous function $u:[0,+\infty) \to \X$ which satisfies the following integral equation
\begin{equation}\label{SolED}
 u(t)=U(t,0)\left( u_0+h(u)\right)+\int_0^t U(t,s)F(s,u(s),\mathcal{B}u(s))ds\, .
\end{equation}
\end{defn}
Now, let us apply the techniques on existence and uniqueness studied before to the integral equation (\ref{SolED}). Before that, let us define $B_0:\R^+ \times \R^+ \times \X \to \X$ 
$$B_0(t,s,x):=B(t,s)x \, ;$$
which comes from the integral representation of the causal operator $\mathcal{B}$, and also define the following constant
$$C_B=\sup_{s\geq 0}\int_0^s||B(s,\tau)||d\tau <\infty \, .$$

\begin{thm}\label{theoaaa1} Let us suppose that $B_0$ satisfies condition {\bf (H5)}, $F \in AAA(\R^+\times \X \times\X;\X)$, and also that conditions {\bf (AT), (bAA)} and {\bf (Ex)} holds. Let 
$\varrho >0$, and define the set $\Delta_0=\{y \in AAA(\R^+,\X):||y-y_0||_{\infty}\leq \varrho\},$ where 
$$ y_0(t):=U(t,0)(u_0+g(0))+\int_{0}^{t}U(t,s)F(s,0,0)ds\, .$$
Suppose that $||y_0||_{\infty}\leq \varrho$, there exists positive constants $L_F, L_g $ such that $g: AAA(\R^+; \X)\to \X $ is $L_g$-Lipschitz in $\Delta_0$, and
$$||F(t,y(t),z(t))-F(t,x(t),w(t))||\leq L_F\left(||y(t)-x(t)||+  ||z(t)-w(t)||\right) ,\ y,x,z,w \in \Delta_0, t\in\R\, .$$
Furthermore, the following inequality holds:
\begin{equation}\label{eq33}
ML_g+\dfrac{M}{\delta}\left(1+C_B \right) L_F \leq \dfrac{\varrho}{\varrho +||y_0||_{\infty}}\, .
\end{equation}
Then equation (\ref{eqX01})- (\ref{eqXX01}) has a unique asymptotically almost automorphic mild solution.
\end{thm}

\begin{proof}
Let us define the following operator:
\begin{eqnarray}\label{EQITD03}
\Gamma: AAA(\R^+;\X) &\to& AAA(\R^+;\X)\\ 
\Gamma y(t)&=& U(t,0)(u_0+g(y))+\int_{0}^{t}U(t,s)F(s,y(s),\mathcal{B}y(s))ds \nonumber
\end{eqnarray}
From Lemmas \ref{l22}, \ref{l23} we have that $\Gamma$ is a well defined operator. Let us prove that 
 $\Gamma$ has a fixed point in $\Delta_0$.\\
1) Let $y \in \Delta_0$ then:
\begin{eqnarray*}
||\Gamma y(t)-y_0(t)||&\leq & ||U(t,0)\left( g(y)-g(0)\right)|| + \int_{0}^{t}||U(t,s)||\, ||F(s,y(s),\mathcal{B}y(s))-F(s,0,0)||ds\\
&\leq& ML_g||y||_{\infty}+L_F \int_{0}^{t}||U(t,s)|| \left( ||y(s)||+||\mathcal{B}y(s)|| \right) ds\\
&\leq& \left(ML_g+ \dfrac{M}{\delta}(1+C_B)L_F \right) (\varrho+||y_0||_{\infty})\\
&\leq&\varrho \, .
\end{eqnarray*}
That is, $\Gamma(\Delta_0)\subseteq\Delta_0$\, .\\

\noindent 2) Let $y_1,y_2 \in \Delta_0$ then:
\begin{eqnarray*}
||\Gamma y_1(t)-\Gamma y_2(t)||&\leq & \int_{0}^{t}||U(t,s)||\, ||F(s,y_1(s), \mathcal{B}y_1(s))-F(s,y_2(s),\mathcal{B}y_2(s))||ds\\
&+& ||U(t,0) \left( g(y_1)+g(y_2)\right) || \\
&\leq& ML_g||y_1-y_2||_{\infty} + L_F (1+C_B) \int_{0}^{t}||U(t,s)||ds||y_1-y_2||_{\infty}\\
&\leq& \left( ML_g+ L_F (1+C_B) \dfrac{M}{\delta} \right) ||y_1-y_2||_{\infty}.
\end{eqnarray*}
Therefore, $\Gamma$ is a contraction on $\Delta_0$ and the conclusion follows from the Banach's contraction Theorem. 
\end{proof}



From Theorem \ref{teos2} we can obtain for the equation (\ref{eqX}) the following result
\begin{thm}\label{theoaaa12}
Suppose that conditions {\bf (AT), (bAA)} and {\bf (Ex)} holds,  $F \in AAA(\R^+\times\X \times\X;\X)$ is $L_F$-Lipschitz. 
Let $\rho>0$ and the set 
$\Delta_0=\{y\in AAA(\R^+;\X): ||y-y_0||_{\infty}\leq\rho\},$ where 
 $$y_0(t)=U(t,0)(u_0+g(0))+\int_0^tU(t,s)F(s,0,0)ds\, .$$
Further suppose that $0<\theta=(1-\xi_0)^{-1}||\Gamma y_0-y_0||_{\infty}\leq\rho ,$ where $\Gamma$ is defined en (\ref{EQITD03}), $g: AAA(\R^+; \X)\to \X $ is $L_g$-Lipschitz in $\Delta_0$ and 
$$\xi_0:=ML_g+\dfrac{M}{\delta}L_F(1+C_B)\, .$$
 Then equation (\ref{eqX01})- (\ref{eqXX01}) has a unique mild solution  $y\in AAA(\R^+;\X)$ such that
$||y-y_0||_{\infty}\leq \rho.$\\
\end{thm}
\begin{proof}
It follows from the hypothesis that the operator $\Gamma: AAA(\R^+ ; \X) \to  AAA(\R^+ ; \X)$, such that 
$$\Gamma y(t)=U(t,0)(u_0+g(y))+\int_0^t U(t,s)F(s,y(s),\mathcal{B}y(s))ds\, ,$$
is well defined. Now, consider the following closed set
$$\Theta_0=\{y \in AAA(\R^+; \X)\, : \, ||y-y_0||_{\infty} \leq \theta\}\, ,$$
and let us look for a fixed point of $\Gamma$ in $\Theta_0$.

\noindent First note that if $y \in \Theta_0$, then 
\begin{eqnarray*}
||\Gamma y(t)-y_0(t)|| &\leq & ||\Gamma y(t)- \Gamma y_0(t)||+ ||\Gamma y_0(t)-y_0(t)|| \\
&\leq & ||\Gamma y_0(t)-y_0(t)||+ ||U(t,0) \left( g(y)-g(y_0) \right)||+\\
&+& \int_0^t ||U(t,s)\left( F(s,y(s), \mathcal{B}y(s))-F(s,y_0(s), \mathcal{B}y_0(s)) \right) || ds\\
& \leq & ML_g ||y-y_0||_{\infty}+L_F\dfrac{M}{\delta} \left( 1+C_B\right)||y-y_0||_{\infty}+ ||\Gamma y_0-y_0||_{\infty}\\
& \leq & \left( ML_g+L_F\dfrac{M}{\delta}(1+C_B)\right)||y-y_0||_{\infty}+ ||\Gamma y_0-y_0||_{\infty}\\
&\leq & \theta \, .
\end{eqnarray*}
This implies that $\Gamma (\Theta_0) \subseteq \Theta_0$. On the other hand, evidently the operator $\Gamma$ is Lipschitz with Lipschitz's constant $\xi_0<1$.
\end{proof}

The proof of the following theorem is the same as \cite[Theorem 2.7]{16}; in fact, the result is a little generalization of the cited one.

\begin{thm}\label{th33}
Suppose that conditions {\bf (AT), (bAA)} and {\bf (Ex)} holds,  $F \in AAA(\R^+\times\X \times\X;\X)$, $g:C(\mathbb{R}^+, \X) \to \X$ and there exists functions $L_g, \ L_F :\R^+ \to \R^+$  such that:
\begin{eqnarray}\label{eq35}
||F(t,x,y)-F(t,z,w)|| & \leq& L_F(r)\left( ||x-z|| + ||y-w||\right)\\
||g(x)-g(y)||& \leq & L_g(r)  ||x-y||, \, \nonumber
\end{eqnarray}
for all $ \ x,y,z,w \in \overline{B}(0,r) $ (the closed ball of radius r). 
Let $C=\sup_{t\geq 0}||F(t,0,0)||$; if we have the inequality:
\begin{equation}\label{eq36}
\sup_{r>0}(\dfrac{\delta r}{M}-\delta rL_g(r)-rL_F(r)(1+C_B))>C+\delta\left( ||y_0||+||g(0)||\right) ,
\end{equation}
then equation (\ref{eqX01})- (\ref{eqXX01}) has a unique asymptotically almost automorphic mild solution.
\end{thm}

\begin{cor}\label{corth33}
Suppose that conditions of theorem \ref{th33} holds and that $L_F(\cdot)=L_F$ and $L_g(\cdot)=L_g$ are positive constants. 
If we have the inequality 
$$\dfrac{\delta }{M}>\delta L_g +(1+C_B)L_F\, , $$
then equation (\ref{eqX01})- (\ref{eqXX01}) has a unique asymptotically almost automorphic mild solution.
\end{cor}
\subsection{A particular case: autonomous integro-differential equation. \\}

Let us give some insights of the following particular version of the integro-differential equation (\ref{eqX01})- (\ref{eqXX01}).
\begin{eqnarray}
u'(t)&=& Au(t)+ \int_0^t B(t-s)u(s)ds+f(t,u(t)),\ t\geq0\; , \label{eqAX01}\\
u(0)&= & u_0 + g(u)\, ; \label{eqAXX01}
\end{eqnarray}
where $u_0 \in \X$, $A$ and $B(t),\ t \geq  0$ are linear, closed and densely defined operators on the Banach space $\X$; and $A,B(t),f,g$ satisfies appropriate conditions.

Existence and uniqueness of the asymptotically almost automorphic solution to the autonomous integro-differential equation (\ref{eqAX01})-(\ref{eqAXX01}) has been studied in \cite{16}. Here, we propose two different theorems to the ones presented in \cite{16}. First let us consider some basic definitions and assumptions: 
%

\begin{defn}\label{d14}\ {\rm A family \{$R(t):t\geq0$\} of continuous linear operators on $\X$ is called a resolvent
operator for the equation \ (\ref{eqAX01}) \ if, and only if:}
\begin{enumerate}
\item [{\rm (R1)}] {\rm $R(0)=I$, is the identity operator on $\X$.}
\item [{\rm (R2)}] {\rm for all $x\in \X$, the operator $t \rightarrow R(t)x$ is a continuous function on
$[0,+\infty)$.}
\item [{\rm (R3)}] {\rm For all $t\geq0$ the operator $R(t)$ is continuous on $Y$, and for all $y\in Y$, the application $t\rightarrow R(t)y$ belongs to $C([0,+\infty);Y)\cap C^1([0,+\infty);\X)$  and satisfies}
$$\frac{d}{dt}R(t)y=AR(t)y+\int_0^tB(t-s)R(s)yds=R(t)Ay+\int_0^tR(t-s)B(s)yds ,\ t\geq0, $$
\end{enumerate}

\noindent {\rm where $Y=D(A)=B(t)$ for all $t\geq 0$ and is equipped with the graphic norm. 
We refer to \cite{19,JPSS} to find details on resolvent operators and same conditions on their existence.}
\end{defn}

\noindent If the resolvent operator  $R(\cdot)$ of the equation (\ref{eqAX01}) exist, then the mild solution of equations (\ref{eqAX01})-(\ref{eqAXX01}) is defined as follows:
\begin{defn}\label{d15}\ {\rm A function $u\in C(\R^+;\X)$ is called a mild solution of the equation
(\ref{eqAX01})-(\ref{eqAXX01}) if}
$$u(t)=R(t)\left( u_0+g(u) \right)+\int_0^tR(t-s)f(s,u(s))ds,\ t\geq 0 \, .$$
\end{defn}

%

\begin{defn}{\bf (Uniform exponential stability).} {\rm 
The resolvent operator  $(R(t))_{t\geq0}$ has uniform exponential stability, if there exists positive constants
$M,\delta>0$ such that for all $t\geq0$ we have $\|R(t)\|\leq Me^{-\delta t}$.}
\end{defn}
\noindent The following technical condition will be needed:\\

\noindent {\bf (A)} $(R(t))_{t\geq0}$ has uniform exponential stability.\\

\begin{thm}\label{th31} Let $f \in AAA(\R^+\times\X;\X)$, $g: AAA(\R^+; \X)\to \X$, $\{R(t)\}_{t\geq0}$ the resolvent operator which satisfies condition {\bf (A)}, $\rho >0$, and the set $\Delta_0=\{y \in AAA(\R^+,\X):||y-y_0||_{\infty}\leq \rho\},$ where 
$$ y_0(t)=R(t)(u_0+g(0))+\int_{0}^{t}R(t-s)f(s,0)ds.$$
Suppose that $||y_0||_{\infty}\leq \rho$, there exists positive constants $L_f , L_g$ such that:
$$||f(t,y(t))-f(t,x(t))||\leq L_f||y(t)-x(t)||,\ y,x \in \Delta_0, t\in\R^+\, ,$$
and $g:AAA(\R^+;\X)\to \X$ is $L_g$-Lipschitz in $\Delta_0$. If the constants $L_g, L_f,\rho, M,\delta$ satisfies the inequality:
\begin{equation}\label{eq33}
\delta L_g+L_f < \dfrac{\rho \delta}{M(\rho+||y_0||_{\infty})}\, ;
\end{equation}
then, equation (\ref{eqAX01})-(\ref{eqAXX01}) has a unique mild solution in $\Delta_0$.
\end{thm}
We comment that, in the proof of theorem \ref{th31}, it is crucial the operator  $\Gamma: AAA(\R^+;\X) \to AAA(\R^+;\X)$ such that

\begin{equation}\label{finop}
\Gamma u(t)=R(t)\left( u_0+g(u) \right)+\int_0^tR(t-s)f(s,u(s))ds\, .
\end{equation}

\noindent Our last abstract result is the following:
\begin{thm}\label{th313}
Let $\Gamma$ be the operator defined in (\ref{finop}), $\{R(t)\}_{t\geq 0}$ the resolvent operator which satisfies
condition {\bf (A)}. Let $\rho>0$ and the set 
$\Delta_0=\{y\in AAA(\R^+;\X): ||y-y_0||_{\infty}\leq\rho\},$ where 
 $$y_0(t)=R(t)(u_0+g(0))+\int_0^t R(t-s)f(s,0)ds\, .$$
Further suppose $f,g \in AAA(\R^+\times\X;\X)$ are $L_f, L_g$-Lipschitz (respectively) for the second variable in $\Delta_0$. If  the inequality 
$$0<\theta=\left( 1-ML_g-\dfrac{M}{\delta}L_f \right) ^{-1}||\Gamma y_0-y_0||_{\infty}\leq\rho\, $$
holds; then, equation (\ref{eqAX01})-(\ref{eqAXX01}) has a unique mild solution  $y\in \Delta_0$ . 
\end{thm}
%

\section{Some applications}\label{appl}

\subsection{Heat conduction in materials with memory.\\\\}


Let $\Omega \subset \R^3$ be an open, connected and bounded set, with $\partial\Omega$ its $C^{\infty}$ boundary. The
conduction of the heat in materials with memory is described with the following partial integro-differential equation:

\begin{eqnarray}\label{EQ1}
\frac{\partial^2 \theta}{\partial^2 t}(x,t)+\beta(0)\frac{\partial \theta}{\partial
t}(x,t)&=&\alpha(0)\Delta\theta (x,t)-\int_{-\infty}^{t}\beta'(t-s)\frac{\partial \theta}{\partial
t}(x,s)ds+\\
&+&\int_{-\infty}^{t}\alpha'(t-s)\Delta\theta(x,s)ds+a(t)b(\theta(t)),\nonumber
\end{eqnarray}
where $\alpha, \beta \in C^2([0,+\infty[;\R)$, are the thermal relaxation function of the heat flux and the energy
relaxation function respectively, with $\alpha (0),\beta (0)$ positives and $\Delta$ is the Laplace operator in $\Omega$.
If the material is isotropic, the temperature $\theta(x,t)$ not depends on the position  $x\in \Omega$ and is know for
all  $t\leq 0$, then (\ref{EQ1}) has the form:

\begin{eqnarray}\label{EQ2}
\theta^{''}(t)+\beta(0)\theta^{'}(t)&=&\alpha(0)\Delta\theta (t)-\int_{0}^{t}\beta'(t-s)\theta^{'}(s)ds+\\
&+&\int_{0}^{t}\alpha'(t-s)\Delta\theta(s)ds+a(t)b(\theta(t)).\nonumber
\end{eqnarray}
If we introduce the new function $\eta(t)=\theta'(t)$, then the previous ones give us the system:
\begin{eqnarray*}
 \theta'(t) &=& \eta(t).\\
\eta'(t)&=&-\beta(0)\eta(t)+\alpha(0)\Delta\theta (t)-\int_{0}^{t}\beta'(t-s)\eta(s)ds+
\int_{0}^{t}\alpha'(t-s)\Delta\theta(s)ds+a(t)b(\theta(t)),
\end{eqnarray*}
and in matricial form is:

\begin{eqnarray*}
\begin{bmatrix}
 {\theta(t)}\\
{\eta(t)}
\end{bmatrix}^{'}
=\begin{bmatrix}
{0}&{I}\\
{\alpha (0)\Delta}&{-\beta(0)I}
\end{bmatrix}
\begin{bmatrix}
 {\theta(t)}\\
{\eta(t)}
\end{bmatrix}+
\int_{0}^{t}\begin{bmatrix}
{0}&{0}\\
{\alpha'(t-s)\Delta}&{-\beta'(t-s)I}
\end{bmatrix}\begin{bmatrix}
 {\theta(s)}\\
{\eta(s)}
\end{bmatrix}ds+\begin{bmatrix}
{0}\\
{a(t)b(\theta(t))}.
\end{bmatrix}.
\end{eqnarray*}
Let us consider the ambient space $X=H_0^1(\Omega)\times L^2(\Omega)$ and identify the operators:
\begin{eqnarray*}
 A=\begin{bmatrix}
{0}&{I}\\
{\alpha (0)\Delta}&{-\beta(0)I}
\end{bmatrix} , B(t-s)=\begin{bmatrix}
{0}&{0}\\
{\alpha'(t-s)\Delta}&{-\beta'(t-s)I}
\end{bmatrix},
\end{eqnarray*}
furthermore for $\begin{bmatrix}
{\theta}\\
{\eta}
\end{bmatrix}$ $\in X$ identify:

\begin{eqnarray*}
 u=\begin{bmatrix}
{\theta}\\
{\eta}
\end{bmatrix} , f(t,u)=\begin{bmatrix}
{0}\\
{a(t)b(\theta)}
\end{bmatrix},
\end{eqnarray*}
and for each
$\begin{bmatrix}
  {\theta_0}\\{\eta_0}
 \end{bmatrix} \in X
$ consider the initial condition $u(0)=u_0=\begin{bmatrix}
  {\theta_0}\\{\eta_0}
 \end{bmatrix}$. 
Then we have the following integro-differential equation with nonlocal initial condition:
\begin{eqnarray}\label{EQ3}
  u'(t)&=&Au(t)+\int_0^t B(t-s)u(s)ds+f(t,u(t)), t\in \R^+,\\
 u(0)&=&u_0+h(u),\label{EQ33}
\end{eqnarray}
with $D(A)=(H^2(\Omega)\cap H_0^1(\Omega))\times H_0^1(\Omega)$.

It follows from \cite{GCH} (see also \cite{19}), that $A$ generates a semigroup $\{T(t)\}_{t\geq 0}$ with
$||T(t)||\leq Me^{-\gamma t}$ for all $t\geq 0$ and $M,\gamma$ positive constants . Let $B(t)=F(t)A$, where:

\begin{eqnarray*}
 F(t)=\begin{bmatrix}
{0}&{0}\\
{-\beta'(t)I+\beta(0)\frac{\alpha'(t)}{\alpha(0)}I}&{\frac{\alpha'(t)}{\alpha(0)}I} 
      \end{bmatrix}.
\end{eqnarray*}
Let us consider the conditions:\\

\noindent $R_1)$. $\alpha'(t)e^{\gamma t},\alpha^{''}(t)e^{\gamma t},\beta'(t)e^{\gamma t},\beta^{''}(t)e^{\gamma t}$
 are bounded and uniformly continuous functions. \\
 
\noindent $R_2)$. Let $p,q >1$ such that $1/p+1/q=1$ and for all $ t\geq0$, $$\max\{||F_{21}(t)||,||F_{22}(t)||\}\leq \frac{\gamma e^{-\gamma t}}{pM}, \ \
\max\{||F_{21}'(t)||,||F_{22}'(t)||\}\leq \frac{\gamma^2 e^{-\gamma t}}{(pM)^2}\, .$$
Under $R_1)$ and $R_2)$, R.C. Grimmer in \cite{19} showed that the equation (\ref{EQ3}) 
has a resolvent operator $\{R(t)\}_{t\geq 0}$ which satisfies condition {\bf (A)}, that is  $\forall t \geq0: \
||R(t)||\leq Me^{-\frac{\gamma t}{q}}.$

Now, assume that\\

\begin{enumerate}
\item  $$y_0(t)=R(t)(u_0+h(0))+\int_0^tR(t-s)f(s,0)ds \, .$$
\item  $h(u)=\begin{bmatrix}
  {h_1(\xi)}\\{h_2(\eta)}
 \end{bmatrix} $ , where $h_1: C(\R^+; H_0^1(\Omega))\to H_0^1(\Omega)$ and $h_2: C(\R^+; L^2(\Omega))\to L^2(\Omega)$ are Lipschitz with the same Lipschitz's constant $\dfrac{\rho}{pM(\rho+||y_0||_{\infty})}$.\\
 
\item   $b: H_0^1(\Omega)\to L^2(\Omega)$ is also Lipschitz, with Lipschitz's constant $\dfrac{\gamma \rho}{qM||a||(\rho +||y_0||_{\infty})}$.\\
\end{enumerate} 
With this preliminaries, the following theorem is an application of theorem \ref{th31}:
\begin{thm} Let $a\in AAA(\R^+;\R)$ 
and consider a real number $\rho >0$ such that:
\begin{equation}\label{EQ4}
 \rho \geq M \left( ||u_0||+ ||h(0)||+\dfrac{q}{\gamma}||a||\;||b(0)||\right).
\end{equation}
Suppose that $(1)-(3)$ holds.
Then, equation (\ref{EQ3})-(\ref{EQ33}) has a unique mild solution $y\in AAA(\R^+;\X)$ such that $||y-y_0||_{\infty}\leq \rho$.
\end{thm}
We comment that, under plausible modifications, it is also possible to give the sufficient conditions under which theorem \ref{th313} is applicable to equation (\ref{EQ3})-(\ref{EQ33}).

\subsection{Semilinear parabolic evolution equations with finite delay. \\}

Let $\tau >0$ be a fixed real number. Let us consider the following semilinear evolution equation with delay:
\begin{eqnarray}\label{semieq}
x'(t) &=& A(t)x(t)+f(t,x(t-\tau)), \, \, \, t\, \, \,  \in \mathbb{R}\, .
\end{eqnarray}
In what follows, we assume that the family $\{A(t)\}_{t\in\R}$ satisfies the Acquistapace-Terrini condition {\bf (AT)}, but on the whole real line, that is $J=\R$ in section \ref{SecAT}. Also, it is assumed that conditions {\bf (bAA)} and {\bf (Ex)} holds. Under this basic assumptions, we define the mild solution to equation (\ref{semieq}), as a continuous function $x:\R \to \X$ which satisfies the following integral equation
\begin{equation}\label{mildde}
x(t)=U(t,a)x(a)+\int_a^t U(t,s)f(s,x(s-\tau)) ds\, .
\end{equation}
Note that, since the evolution family is exponentially bounded, then the mild solution satisfy
\begin{equation}\label{smildde}
x(t)=\int_{-\infty}^t U(t,s)f(s,x(s-\tau)) ds\, .
\end{equation}
Moreover, since $U(t,s)\in \mathcal{B}(\X)$\, for $(t,s) \in \triangle_{J}^+:=\{(t,s)\in J \times J\; :\; t\geq s \}\; $, then it is easy to check that a solution to the integral equation (\ref{smildde}), is a solution to (\ref{mildde}).

\begin{thm} Let $f\in AA(\R\times \X; \X)$ be $L_f$-Lipschitz. Let $\rho$ be a positive real number such that 
$$\dfrac{M}{\delta} ||f(\cdot,0)||_{\infty} \leq \rho \, \, \, \, \,and \, \, \, \, 
ML_f< \dfrac{\rho \delta}{\rho+ ||x_0||_{\infty}}\, ;$$
where, 
$$x_0(t)=\int_{-\infty}^tU(t,s)f(s,0)ds\, .$$
Then, there is a unique mild solution $x\in AA(\R;\X)$ to equation (\ref{semieq}) such that $||x-x_0||_{\infty}\leq \rho$ .
\end{thm}
Our last result read as follows:

\begin{thm}\label{theoaaa121final}
Suppose that $f \in AAA(\R^+\times\X;\X)$ is $L_f$-Lipschitz. 
Let $\rho>0$ and the set 
$\Delta_0=\{x\in AA(\R;\X): ||x-x_0||_{\infty}\leq\rho\},$ where 
 $$x_0(t)=\int_{-\infty}^tU(t,s)f(s,0)ds\, .$$
Further suppose that $0<\theta=(1-\dfrac{M}{\delta}L_f)^{-1}||\Gamma x_0-x_0||_{\infty}\leq\rho $, where $\Gamma$ is defined from the right hand side of (\ref{smildde}). 
 Then, equation (\ref{semieq}) has a unique mild solution  $x\in AA(\R;\X)$ such that $||x-x_0||_{\infty}\leq \rho .$
\end{thm}
 
\section{acknowledge.}
M. Pinto has been partially supported by Fondecyt 1170466 and 1120709. A. Ch\'avez has been partially supported by Fondecyt Registro 64044-Per\'u, he also wants to thank the hospitality of his friend Mr. Alvaro Corval\'an Azagra.

\end{document}